\documentclass[13pt]{amsart}
\usepackage{amsmath,amssymb,amsthm,amscd}
\usepackage{graphicx}
\usepackage{color}
\usepackage{tikz}
\numberwithin{equation}{section}

\newtheorem{prop}{Proposition}[section]
\newtheorem{theo}[prop]{Theorem}
\newtheorem{lem}[prop]{Lemma}
\newtheorem{coro}[prop]{Corollary}
\newtheorem{rem}[prop]{Remark}
\newtheorem{exa}[prop]{Example}
\newtheorem{defi}[prop]{Definition}

\def\begeq{\begin{equation}}
\def\endeq{\end{equation}}

\begin{document}
\title[Equivariant $\mathbb R$-test configurations and semistable limits...]{Equivariant $\mathbb R$-test configurations and semistable limits of $\mathbb Q$-Fano group compactifications}
\author[Yan Li and ZhenYe Li]{Yan Li$^{*1}$ and ZhenYe Li $^{*2}$}

\address{$^{*1}$School of Mathematics and Statistics, Beijing Institute of Technology, Beijing, 100081, China.}
\address{$^{*2}$College of Mathematics and Physics, Beijing University of Chemical Technology, Beijing, 100029, China.}
\email{liyan.kitai@yandex.ru,\ \ \ lizhenye@pku.edu.cn}

\thanks {$^{*1}$Partially supported by NSFC Grant 12101043 and the Beijing Institute of Technology Research Fund Program for Young Scholars.}



\begin{abstract}
Let $G$ be a connected, complex reductive group. In this paper, we classify $G\times G$-equivariant normal $\mathbb R$-test configurations of a polarized $G$-compactification. Then  for  $\mathbb Q$-Fano $G$-compactifications, we express the H-invariants of its equivariant normal $\mathbb R$-test configurations in terms of the combinatory datas. Based on \cite{Han-Li}, we compute the semistable limit of a K-unstable Fano $G$-compactification. As an application, we show that for the two smooth K-unstable Fano $SO_4(\mathbb C)$-compactifications, the corresponding semistable limits are indeed the limit spaces of the normalized K\"ahler-Ricci flow.
\end{abstract}
\maketitle

\section{Introduction}

Let $M$ be a Fano manifold, namely, a compact K\"ahler manifold with positive first Chern class $c_1(M)$. Consider the following normalized K\"ahler-Ricci flow:
\begin{align}\label{kahler-Ricci-flow}
\frac{\partial}{\partial t}\omega(t) = -{\rm Ric}(\omega( t))+\omega(t), ~\omega(0)=\omega_0,
\end{align}
where $\omega_0$ and $\omega (t)$ denote the initial K\"ahler metric and the solutions of \eqref{kahler-Ricci-flow}, respectively. Cao \cite[Section 1]{cao} showed that \eqref{kahler-Ricci-flow} always have a global solution $\omega(t)$ for all $t\geq 0$ whenever $\omega_0\in2\pi c_1(M)$. A long-standing problem concerns the limiting behavior of $\omega(t)$ as $t\to \infty$. Tian-Zhu \cite{tianzhu, TZ4} showed that: if $M$ admits a K\"ahler-Ricci soliton, then $\omega(t)$ will converge to it. However, in general $\omega(t)$ may not have a limit on $M$. The famous Hamilton-Tian conjecture (cf. \cite[Section 9]{Ti97}) suggests that any sequence of $\{(M, \omega(t_i))\}_{i\in\mathbb N_+}$ with $t_i\to+\infty$ contains a subsequence converging to a length space $(M_\infty,\omega_\infty)$ in the Gromov-Hausdorff topology, and $(M_\infty,\omega_\infty)$ is a smooth K\"ahler-Ricci soliton outside a closed subset $S$ of (real) codimension at least $4$. Moreover, this subsequence converges  locally to the regular part of $(M_\infty,\omega_\infty)$ in the Cheeger-Gromov topology. This implies that, by taking the limit, the complex structure of $M$ may jump so that under the new complex structure there exists a K\"ahler-Ricci soliton.

The Gromov-Hausdorff convergency follows from Perelman \cite{Pe} and Zhang \cite{Zh1, Zhq}. Tian-Zhang \cite{TZhzh} first confirmed the whole conjecture when ${\rm dim}(M)\leq3$. Chen-Wang \cite{CW} and Bamler \cite{Bam} then solved the remaining higher dimensional cases. In fact, Bamler \cite{Bam} proved a generalized version of the conjecture.

It is then natural to study the regularity of the limit space $(M_\infty,\omega_\infty)$. In fact, Tian-Zhang \cite{TZhzh} proved that $M_\infty$ is a $\mathbb Q$-Fano variety whose singular set coincides with $S$. Concerning the further regularity of $M_\infty$, there was a folklore speculation states that $(M_\infty, \omega_\infty)$ is actually a smooth Ricci soliton, or equivalently, \eqref{kahler-Ricci-flow} always has Type-I solution. In \cite{LTZ},  Li-Tian-Zhu disproved this folklore speculation by constructing examples of Type-II solution of \eqref{kahler-Ricci-flow}. That is, a solution $\{\omega(t)\}_{t\geq 0}$ such that curvature of $\omega(t)$ is not uniformly bounded for $t\in[0,+\infty)$. More precisely, they proved

\begin{theo}\label{singular-type2} \cite[Theorem 1.1]{LTZ} Let $G$ be a connected, complex semisimple Lie group and $K$ be its maximal compact subgroup. Let $M$ be a Fano $G$-compactification which admits no K\"ahler-Einstein metrics.  Then any solution of K\"ahler-Ricci flow (\ref{kahler-Ricci-flow}) on $M$ with $K\times K$-invariant initial metric $\omega_0\in 2\pi c_1(M)$ is of Type-II.
\end{theo}

In particular, the above theorem shows that there are two smooth Fano $SO_4(\mathbb C)$-compactification (see Section 6 below) which involves Type-II solution of \eqref{kahler-Ricci-flow}. To our knowledge, these are the first examples of Type-II solution in the literature. Note that both examples are also K-unstable.

It is of great interest to discover the limits of the K\"ahler-Ricci flow \eqref{kahler-Ricci-flow} on K-unstable Fano group compactifications. According to \cite{TZhzh}, such a limit should be a $\mathbb Q$-Fano variety, admitting (weak) K\"ahler-Ricci solitons. Recall that a K\"ahler-Ricci soliton on a complex manifold $M$ is a pair $(X, \omega)$, where $X$ is a holomorphic vector field on $M$ and $\omega$ is a K\"ahler metric on $M$, such that
\begin{align}\label{kr-soliton}
{\rm Ric}(\omega)-\omega={\rm L}_X(\omega),
\end{align}
where ${\rm L}_X(\cdot)$ is the Lie derivative along $X$. If $X=0$, the K\"ahler-Ricci soliton becomes a K\"ahler-Einstein metric. The uniqueness theorem in \cite{TZ1} states that a K\"ahler-Ricci soliton on a compact complex manifold, if it exists, must be unique modulo ${\rm Aut}(M)$. 
Furthermore, $X$ lies in the center of Lie algebra of a maximal reductive subgroup ${\rm Aut_r}(M)\subset{\rm Aut}(M)$.

It was first showed in \cite[Section 7]{LTZ2} that the limits of \eqref{kahler-Ricci-flow} on the two smooth K-unstable $SO_4(\mathbb C)$-compactifications can not be $SO_4(\mathbb C)$-compactifications any more. In fact, this is the case for Fano compactifications of semisimple groups in general. To see this, note that on a compactifiaction of semisimple group, the center of ${\rm Aut_r}(M)$ is trivial. Hence any K\"ahler-Ricci soliton must be a K\"ahler-Einstein metric. On the other hand, it was showed in \cite{LL} that for any semisimple group $G$, the $\mathbb Q$-Fano $G$-compactioficationsthat admit K\"ahler-Einstein metrics are finite. There are of course infinitely many $\mathbb Q$-Fano compactiofications of $G$. Hence in general the K\"ahler-Ricci flow \eqref{kahler-Ricci-flow} converges to a limit which is no longer a $G$-compactification.

On the other hand, Chen-Sun-Wang \cite{Chen-Wang-Sun} showed the following phenomena: in general, $M$ may be degenerated to $M_\infty$ via two-step degenerations: first by a ``semistable degeneration" to a normal variety $\mathcal X$ with a holomorphic vector field $\Lambda$, and then a ``polystable degeneration" which degenerates $(\mathcal X,\Lambda)$ to $(M_\infty,\Lambda)$. Moreover, the soliton vector field on $M_\infty$ is precisely $\Lambda$ arises in the ``semistable degeneration". It is conjectured that there is an algebraic way to determine these two-step degenerations. 

One approach to determining the ``semistable degeneration" is minimizing the H-invariant. The H-invariant was first introduced by Tian-Zhang-Zhang-Zhu \cite[Section 5]{TZZZ} for holomorphic vector fields in the study of \eqref{kahler-Ricci-flow}. Dervan-Sz\'ekelyhidi \cite{Dervan-Sze} generalized it to special $\mathbb R$-test configurations. Very recently, Han-Li \cite{Han-Li} provided a more precise expression of the H-invariant (cf. \cite[Remark 2.41]{Han-Li}) of any $\mathbb R$-test configuration. Assuming the existence of a special minimizer of the H-invariant, \cite{Han-Li} proved its uniqueness. Thus on a Fano manifold, ``semistable degeneration" is the unique special $\mathbb R$-test configuration which minimizes the H-invariant.\footnote{Our convention of the H-invariant follows \cite{Han-Li} and differs from \cite{Dervan-Sze} by a sign. See Section 2.1 for detail.} Furthermore, the central fibre $(\mathcal X,\Lambda)$ is (modified) K-semistable (cf. \cite[Theorem 1.2]{Han-Li}). We will call it the \emph{``semistable limit"} in the following. Also, it is proved by \cite{Li-Wang-Xu} that $(\mathcal X,\Lambda)$ admits a unique ``polystable degeneration" whose centre $\mathcal X_0$ is modified K-polystable with respect to $\Lambda$. Hence they proved that the two degenerations in \cite{Chen-Wang-Sun}, and consequently the Gromov-Hausdorff limit $M_\infty$ for \eqref{kahler-Ricci-flow}, depend only on the algebraic structure of $M$ (cf. \cite[Corollary 1.4]{Han-Li}).
A bit later Blum-Liu-Xu-Zhuang \cite{Blum-Liu-Xu-Zhuang} gives an algebraic proof of the Hamilton-Tian conjecture for general log Fano pairs. By using a finite generation result, \cite[Theorem 1.2]{Blum-Liu-Xu-Zhuang} proved that the H-invariant admits a unique minimizer among a larger class of filtrations defined by valuation. Moreover, the minimizer is a special $\mathbb R$-test configuration. In particular this proves the existence and uniqueness of ``semistable degeneration". The proof in \cite{Blum-Liu-Xu-Zhuang} uses deep results and abstract constructions in birational geometry from the former works of its authors.


In this paper, we will apply the above algebraic approach to find the limit of \eqref{kahler-Ricci-flow} on Fano group compactifications. In particular, we will find the limit of \eqref{kahler-Ricci-flow} on the two Fano $SO_4(\mathbb C)$-compactifications given in \cite{LTZ}.

Let us recall the conception of group compactifications. Suppose that $G$ is an $n$-dimensional connect, complex reductive group which is the complexification of a compact Lie group $K$. Let $M$ be a projective normal variety. $M$ is called a \emph{(bi-equivariant) compactification} of $G$ (or $G$-compactification for simplicity) if it admits a holomorphic $G\times G$-action with an open and dense orbit isomorphic to $G$ as a $G\times G$-homogeneous space.  $(M, L)$ is called a \emph{polarized compactification} of $G$  if $L$ is a $G\times G$-linearized $\mathbb Q$-Cartier ample line bundle on $M$. In particular, when $K_M^{-1}$ is an ample $\mathbb Q$-Cartier line bundle and $L=K_M^{-1}$, we call $M$ a $\mathbb Q$-Fano $G$-compactification. For more knowledge and examples, we refer the reader to \cite{Timashev-Sbo, AK, Del3}, etc.

As mentioned above, the minimizer of the H-invariant is special. Hence for our purpose, we first study the $G\times G$-equivariant normal $\mathbb R$-test configurations with reduced central fibre. We have the following classification result,
\begin{theo}\label{RTC-classify}
Let $(M,L)$ be a polarized $G$-compactificartion with moment polytope $P_+$. Then the $G\times G$-equivariant $\mathbb R$-test configurations of $(M,L)$ with reduced central fibre are in one-to-one correspondence with $W$-invariant, concave, piecewise-linear functions on $$\overline P=\overline{\cup_{w\in W}w(P_+)}$$ whose domains of linearity consist of convex rational polytopes in $\mathfrak M_\mathbb Q$. Here $W$ denotes the Weyl group of $G$ with respect to a fixed maximal torus and $\mathfrak M$ denotes the lattice of characters of $G$. 
\end{theo}
We refer to the readers Section 2.2 for precise meaning of the notations.  A more precise version of Theorem \ref{RTC-classify} will be proved in Section 4 (see Theorem \ref{G-classify-reduced} below). With the help of Theorem \ref{G-classify-reduced}, we estimate the H-invariant of $G\times G$-equivariant normal $\mathbb R$-test configurations via the combinatorial data. In particular we get a precise formula for the special ones. We will see that the minimizer of the H-invariant among equivariant special $\mathbb R$-test configurations is unique. Combining with the arguments of uniqueness in \cite[Section 6]{Han-Li} and \cite{Blum-Liu-Xu-Zhuang} we find the ``semistable degeneration".

\begin{theo}\label{main-thm-1}
Let $G$ be a reductive Lie group and $M$ a $\mathbb Q$-Fano compactification of $G$. Then there is a unique $G\times G$-equivariant special $\mathbb R$-test configuration $\mathcal F_0$ such that the H-invariant
\begin{align}\label{H-minima}
H(\mathcal F_0)\leq H(\mathcal F),~\forall&\text{ $G\times G$-equivariant normal $\mathbb R$-test configuration $\mathcal F$}.
\end{align}
Moreover, $\mathcal F_0$ is the ``semistable degeneration" of $M$.
\end{theo}

Theorem \ref{main-thm-1} will be proved in Section 5.2. Next we check the (modified) K-(poly)stability of the central fibre $\mathcal X_0$ of $\mathcal F_0$. At least for the two K-unstable Fano $SO_4(\mathbb C)$-compactifications, the corresponding central fibres $\mathcal X_0$ are indeed  (modified) K-polystable. This suggests that the ``polystable degeneration" will be trivial and $\mathcal X_0$ is actually the limit $M_\infty$ of \eqref{kahler-Ricci-flow} on these two examples. Namely,
\begin{theo}\label{SO-4-exa}
Let $M$ be a smooth K-unstable Fano $SO_4(\mathbb C)$-compactification. Then the semistable limit of $M$ coincides with the limiting space of \eqref{kahler-Ricci-flow}.
\end{theo}
We will give the limits in Section 6 in terms of the combinatorial data. Moreover, we will see that the limit of \eqref{kahler-Ricci-flow} on a group compactification is always spherical, although possibly may not be a group compactification. This suggests that the class of spherical varieties forms a ``closed" class when considering the moduli problem, while the class of group compactifications usually does not.


The paper is organized as follows: in Section 2 we recall $\mathbb R$-test configurations, H-invariant as well as theory of spherical varieties. In Section 3 we overview the usual $G\times G$-equivariant normal $\mathbb Z$-test configurations of a $G$-compactification. In particular we study the structure of the corresponding central fibres. In Section 4 we classify the  $G\times G$-equivariant normal $\mathbb R$-test configurations by a purely algebraic argument. In particular we classify those with reduced central fibre. In Section 5 we estimate the H-invariant of $G\times G$-equivariant normal $\mathbb R$-test configurations via the combinatorial data (see Theorem \ref{H-f-reduction}). Then we find the minimizer of H-invariant and prove Theorem \ref{main-thm-1}. Finally we test the K-stability of the central fibre (see Proposition \ref{m-semi-stab}). In Section 6 we apply the above results to smooth K-unstable Fano $SO_4(\mathbb C)$-compactifications and prove Theorem \ref{SO-4-exa}. 

\subsection*{Acknowledgement} We are grateful to Prof. Gang Tian for his interests in this paper and sharing his insights on K\"ahler geometry. We would also like to thank Prof. D. A. Timash\"ev for kindly introducing us his book \cite{Timashev-book}, Prof. Chi Li and Feng Wang for helpful discussions and comments. Finally, we thank the referees for their helpful suggestions to an early version of this paper.

\section{Preliminaries}

\subsection{Filtrations and test configurations}
In this section we recall some basic material concerning filtrations and test configurations. We refer to the readers \cite[Section 2.2]{Dervan-Sze} and \cite[Section 2]{Han-Li} for further knowledge.

Let $M$ be a projective variety and $L$ a very ample line bundle over $M$ so that $|L|$ gives a Kodaira embedding of $M$ into projective space. The homogenous coordinate ring (Kodaira ring) of $M$ is
$R(M,L) =\oplus_{k\in\mathbb N}R_k,~\text{where}~R_k=H^0(M,L^k).$
\begin{defi}\label{filtrantion-def}
A filtration $\mathcal F$ of $R$ is a family of subspaces $\{\mathcal F^sR_k\}_{s\in\mathbb R,k\in\mathbb N}$ of $R(M,L)=\oplus_{k\in\mathbb N}R_k$ such that
\begin{itemize}
\item[(1)] $\mathcal F$ is decreasing: $\mathcal F^{s_1}R_k\subset \mathcal F^{s_2}R_k,~\forall s_1\geq s_2\text{ and }k\in\mathbb N;$
\item[(2)] $\mathcal F$ is left-continuous: $\mathcal F^sR_k=\cap_{t<s}\mathcal F^tR_k,~\forall k\in\mathbb N;$
\item[(3)] $\mathcal F$ is linearly bounded: There are constants $c_\pm\in\mathbb Z$ such that for each $k\in\mathbb N$, such that $$\mathcal F^sR_k=0,~\forall s>c_+k\text{ and }\mathcal F^sR_k=R_k,~\forall s<c_-k;$$
\item[(4)] $\mathcal F$ is multiplicative: $\mathcal F^{s_1}R_{k_1}\cdot\mathcal F^{s_2}R_{k_2}\subset\mathcal F^{s_1+s_2}R_{k_1+k_2},$ for all $k_1,k_2\in\mathbb N$ and $s_1,s_2\in\mathbb R.$
\end{itemize}
\end{defi}
Let $\Gamma(\mathcal F,k)$ be the set of values of $s$ where the filtration of $R_k$ is discontinuous. Set
\begin{align}\label{group-discon}
\Gamma_+(\mathcal F):=\cup_k(\Gamma(\mathcal F,k)-\min\Gamma(\mathcal F,k)),
\end{align}
and $\Gamma(\mathcal F)$ the Abelian group generated by $\Gamma_+(\mathcal F)$. We associate to each filtration $\mathcal F$ two graded algebra,
\begin{defi}\label{def-rees-gr}
\begin{itemize}
\item[(1)]
The Rees algebra,
\begin{align}\label{rees-def}
{\rm R}(\mathcal F):=\oplus_{k\in\mathbb N}\oplus_{s\in\Gamma(\mathcal F)+\min\Gamma(\mathcal F,k)}t^{-s}\mathcal F^sR_k,
\end{align}
and
\item[(2)]
The associated graded ring of $\mathcal F$,
\begin{align}\label{GrF-def}{\rm Gr}
(\mathcal F):=\oplus_{k\in\mathbb N}\oplus_{s\in\Gamma(\mathcal F)+\min\Gamma(\mathcal F,k)}t^{-s}(\mathcal F^sR_k/\mathcal F^{>s}R_k).
\end{align}
\end{itemize}
\end{defi}

There is an important class of filtrations, called the $\mathbb R$-test configurations, which can be considered as a generalization of the usual ($\mathbb Z$-)test configuration introduced in \cite{Do}.

\begin{defi}
When ${\rm R}(\mathcal F)$ is finitely generated, we say that $\mathcal F$ is an $\mathbb R$-test configuration of $(M,L)$. In this case, ${\rm Gr}(\mathcal F)$ is also finitely generated. The projective scheme
$$\mathcal X_0:={\rm Proj}({\rm Gr}(\mathcal F))$$
is called the central fibre of $\mathcal F$.
\end{defi}
When $\mathcal F$ is an $\mathbb R$-test configuration, the Abelian group $\Gamma(\mathcal F)$ generated by \eqref{group-discon} has finite rank. Denote its rank by $r_\mathcal F$. Then the total space $\mathcal X$ of $\mathcal F$ is
\begin{align}\label{total-space-ring}
\mathcal X={\rm Proj}_{\mathbb C^{r_\mathcal F}}({\rm R}(\mathcal F)).
\end{align}
Also the $\Gamma(\mathcal F)$-grading of ${\rm Gr}(\mathcal F)$ corresponds to a (possibly real) holomorphic vector field $X$ on $\mathcal X_0$, which generates a rank $r_\mathcal F$ torus (denote by $\mathbb T$) action. Note that we take the convention that the $\exp{(tX)}$-action has weight $t^s$ on the $(\mathcal F^sR_k/\mathcal F^{>s}R_k)$-piece in \eqref{GrF-def}.\footnote{This differs from \cite[Definition 2.12]{Han-Li} by a sign.} We call $X$ the vector field induced by $\mathcal F$.

\begin{rem}\label{F-normalized}
By finite generation, $\mathcal F$ is generated by $\mathcal FR_{k_0}$ for some $k_0\in\mathbb N_+$. By a shifting of $\mathcal F$, we can always normalize $\min\Gamma(\mathcal F,k_0)=0$. Then $\Gamma (\mathcal F)$ contains all points of discontinuity of $\mathcal F$. In the following, all $\mathbb R$-test configurations are assumed to be normalized in this way.
\end{rem}

\begin{rem}
When ${\rm rank}(\Gamma(\mathcal F))=1$, we can embed $\Gamma(\mathcal F)$ in $\mathbb Z$. The above $\mathbb R$-test configuration is simply the usual ($\mathbb Z$-)test configuration introduced in \cite[Definition 2.1.1]{Do}.
\end{rem}

There is an important subclass of normal $\mathbb R$-test configurations,
\begin{defi}
When $L=K_M^{-1}$, an $\mathbb R$-test configuration $\mathcal F$ is called special if the central fibre $\mathcal X_0$ is $\mathbb Q$-Fano and ${\rm Gr}(\mathcal F)$ is isomorphic to $R(\mathcal X_0,-K_{\mathcal X_0})$, the Kodaira ring of $\mathcal X_0$.
\end{defi}

\subsubsection{Equivariant $\mathbb R$-test configurations}
Let $(M,L)$ be a polarized variety with a group $\mathfrak G$-action. Then $\mathfrak G$-acts on its Kodaira ring. Let $\mathcal F$ be a filtration on $R$. Define the action of $\mathfrak G$ on $\mathcal F$ by
\begin{align}\label{group-act-F}
(\sigma\cdot \mathcal F)^sR_k:=\sigma(\mathcal F^sR_k),~\forall s\in\mathbb R,k\in\mathbb N,
\end{align}
for any $\sigma\in \mathfrak G$. Clearly $\sigma\cdot \mathcal F$ is also a filtration on $R$, and it an $\mathbb R$-test configuration if and only if $\mathcal F$ is. As a generalization of equivariant $\mathbb Z$-test configurations, we define

\begin{defi}
A filtration $\mathcal F$ is called {$\mathfrak G$-equivariant} if $\mathcal F^sR_k$ is a $\mathfrak G$-invariant space of $R_k$ for any $s\in\mathbb R$ and $k\in\mathbb N$. That is
\begin{align}\label{group-eq-F}
\sigma(\mathcal F^sR_k)=\mathcal F^sR_k,~\forall s\in\mathbb R,k\in\mathbb N.
\end{align}
\end{defi}

As in the case of equivariant $\mathbb Z$-test configurations, it is clear that when $\mathcal F$ is a $\mathfrak G$-equivariant $\mathbb R$-test configuration, the projection from $\mathcal X$ to the base $\mathbb C^{r_{\mathcal F}}$ is $\mathfrak G$-equivariant. Note that here $\mathfrak G$ acts on $\mathbb C^{r_{\mathcal F}}$ trivially.

\subsubsection{Filtrations and semi-valuations}
Let $\mathcal F$ be a filtration on $R$. For any $\sigma_k\in R_k,~k\in\mathbb N_+$, set
\begin{align*}
 {\mathfrak v}_\mathcal F(\sigma_k):=\max\{s|\sigma_k\in\mathcal F^sR_k\},
\end{align*}
and for any $\sigma:=\sum_{k\in\mathbb N_+}\sigma_k$ with $(0\not=)\sigma_k\in R_k$, set
\begin{align*}
{\mathfrak v}_\mathcal F(\sigma):=\min\{\bar v_\mathcal F(\sigma_k)|0\not=\sigma_k\in R_k\}.
\end{align*}
Then ${\mathfrak v}_\mathcal F(\cdot)$ defines a semi-valuation on $R$ (cf. \cite[Section 2.2]{Han-Li}). Conversely, given any valuation ${\mathfrak v}$ with finite log-discrepancy, there is a filtration $\mathcal F_{({\mathfrak v})}$ on $R$ by (cf. \cite[Example 2.2]{Han-Li}) satisfying the above two relations,
$$\mathcal F_{({\mathfrak v})}^sR_k:=\{\sigma\in R_k|{\mathfrak v}(\sigma)\geq s\}.$$

It is known that when ${\rm Gr}(\mathcal F)$ is an integral domain, ${\mathfrak v}_\mathcal F$ is a valuation and $\mathcal F=\mathcal F_{( {\mathfrak v}_\mathcal F)}$, up to a shifting. In particular, this applies to special $\mathbb R$-test configurations (cf. \cite[Lemma 2.11]{Han-Li}).

\subsection{The H-invariant}
Given a $\mathbb Q$-Fano variety $M$ and $\mathcal F$  a special test configuration of $(M,K_M^{-1})$. Let $\mathcal X_0$ be the central fibre and $X$ the vector field induced by $\mathcal F$. Assume that $\mathcal X_0$ has klt-singularities. Choose any smooth K\"ahler metric $\omega\in2\pi c_1(\mathcal X_0)$ and fix a normalized Ricci potential $h$ of $\omega$. Let $\theta_X(\omega)$ be any potential of $X$ with respect to $\omega$. Tian-Zhang-Zhang-Zhu \cite[Section 5]{TZZZ} defined the following H-invariant of $\mathcal F$,

\begin{align}\label{H-inv-ana}
H(\mathcal F)=V\ln\left(\frac1V\int_{\mathcal X_0}e^{\theta_X(\omega)}\omega^n\right)-\int_{\mathcal X_0}\theta_X(\omega)e^h\omega^n,
\end{align}
Note that $H(\mathcal F)$ is well-defined since different choices of $\theta_X(\omega)$ only differs from each other by a constant.

Let $\mathcal F$ be an $\mathbb R$-test configuration of $(M,K_M^{-1})$. Dervan-Sz\'ekelyhidi \cite{Dervan-Sze} generalized \eqref{H-inv-ana} to $\mathbb R$-test configurations via an algebraic aspect. This is recently modified by Han-Li \cite{Han-Li}. Recall the definition in \cite[Section 2.5]{Han-Li}. First we associate to any $\mathbb R$-test configuration $\mathcal F$ two non-Archimedean functionals. Let $M_{\mathbb Q}^{\rm div}$ be the set of $\mathbb Q$-divisorial valuations of $M$. Denote by $\phi_\mathcal F,\phi_0$ the non-Archimedean metric of $\mathcal F$ and the trivial test configuration (cf. \cite[Section 6]{Boucksom-Hisamoto-Jonsson}), respectively. Set
\begin{align}\label{L-inv-NA}
L^{\rm NA}(\mathcal F):=\inf_{ \mathfrak v\in M_{\mathbb Q}^{\rm div}}(A_M( \mathfrak v)+(\phi_\mathcal F-\phi_0)( \mathfrak v)),
\end{align}
where $A_M(\cdot)$ is the log-discrepancy of a divisor of $M$. Also, denote by $\Delta(\mathcal F^{(t)})$ the Okounkov body of the linear series (cf. \cite[Section 2.4]{Han-Li}),
\begin{align}\label{okounkov-body-layer}
\mathcal F^{(t)}:=\{\mathcal F^{tk}R_k\}_{k\in\mathbb N_+}.
\end{align}
By Definition \ref{filtrantion-def} (3), we see that when $t\ll0$, $\Delta(\mathcal F^{(t)})$ is just the Okounkov body $\Delta$ of $(M,K_M^{-1})$ introduced in \cite{Okounkov},  and $\Delta(\mathcal F^{(t)})=\{O\}$ when $t\gg1$. Define a function $G_\mathcal F:\Delta\to\mathbb R$ by
\begin{align}\label{okounkov-body-func}
G_\mathcal F(z):=\sup\{t|z\in\Delta(\mathcal F^{(t)})\},~z\in\Delta
\end{align}
and set
\begin{align}\label{S-inv-NA}
S^{\rm NA}(\mathcal F):=-\ln\left(\frac{n!}{V}\int_{\Delta}e^{-G_\mathcal F(z)}dz\right).
\end{align}
The H-invariant is given by
\begin{defi}Let $\mathcal F$ be an $\mathbb R$-test configuration of $(M,K_M^{-1})$. Then the H-invariant of $\mathcal F$ is
\begin{align}\label{H-inv-NA}
H(\mathcal F):=L^{\rm NA}(\mathcal F)-S^{\rm NA}(\mathcal F).
\end{align}
\end{defi}
\cite[Example 2.32]{Han-Li} shows that \eqref{H-inv-NA} coincides with \eqref{H-inv-ana} on special test configurations.

\subsubsection{Equivariance of the minimizer}
Let $M$ be a $\mathbb Q$-Fano variety with a $\mathfrak G$-action. We want to show that the H-invariant is $\mathfrak G$-invariant and derive the equivariance of its minimizer. First, we have

\begin{prop}
Let $M$ be a $\mathbb Q$-Fano variety with a $\mathfrak G$-action and $\mathcal F$ be an $\mathbb R$-test configuration on $K^{-1}_M$. Then
\begin{align}\label{H-inv-by-G}
H(\sigma\cdot\mathcal F)=H(\mathcal F),~\forall \sigma\in\mathfrak G.
\end{align}
\end{prop}

\begin{proof}
Recall that \eqref{H-inv-NA}. It suffices to prove both $L^{\rm NA}(\cdot)$ and $S^{\rm NA}(\cdot)$ are $\mathfrak G$-invariant. By \eqref{group-act-F}, since $(\sigma\cdot F)^sR_k$ is always isomorphic to $F^sR_k$, in particular, the linear series $(\sigma\cdot\mathcal F)^{(t)}\cong\mathcal F^{(t)}$, we have $$S^{\rm NA}(\sigma\cdot\mathcal F)=S^{\rm NA}(\mathcal F).$$

Recall \eqref{H-inv-NA}.
Note that group $\mathfrak G$ can act on $M_\mathbb Q^{div}$ by
$$(\sigma\cdot \mathfrak v)(f)= \mathfrak v(f(\sigma^{-1}\cdot))$$ for any rational function $f$ on $M$.
Thus by the definiton of non-Archemidean metric (cf. \cite[Definition 2.17]{Han-Li}), the second term in \eqref{L-inv-NA} satisfies that
$$(\phi_{\sigma\cdot\mathcal F}-\phi_{\rm triv})(\sigma\cdot \mathfrak  v)=(\phi_{\mathcal F}-\phi_{\rm triv})( \mathfrak v)$$
for any valuation $ \mathfrak v$.

It remains to check
$$A_M( \mathfrak v)=A_M(\sigma\cdot \mathfrak  v).$$
Suppose that $\mu:X\to M$ is a resolution such that $ \mathfrak v=c{\rm ord}_D(\cdot)$ for some $c\in\mathbb Q_+$ and $D$ a prime divisor of $X$, i.e., $ \mathfrak v(f)=c{\rm ord}_D(\mu^*\circ f)$.
Then
$$A_M( \mathfrak v)=c(1+{\rm ord}_D(K_X-\mu^*K_M)).$$

Now $(\sigma\cdot  \mathfrak v)(f)=c{\rm ord}_D(\mu^* \sigma^{-1*}\circ f)$.
So for morphism $\mu_1=\sigma^{-1}\circ \mu$, we have
 $(\sigma\cdot \mathfrak  v)(f)=c{\rm ord}_D(\mu_1^*\circ f)$. So
$$A_M(\sigma\cdot \mathfrak  v)=c(1+{\rm ord}_{D}(K_X-\mu_1^*K_M)).$$

Since ${\sigma^{-1}}^*K_M=K_M$, we can compute $A_M(\sigma\cdot \mathfrak  v)$ as
$$A_M(\sigma\cdot \mathfrak  v)=c(1+{\rm ord}_{D}(K_X-\mu^*{\sigma^{-1}}^*K_M))=A_M( \mathfrak v).$$

Now we have
\begin{align*}
L^{\rm NA}(\sigma\cdot\mathcal F)=&\inf_{ \mathfrak v\in M^{\rm div}_\mathbb Q}(A_M( \mathfrak v)+(\phi_{\sigma\cdot\mathcal F}-\phi_{\rm triv})( \mathfrak v))\\
=&\inf_{ \mathfrak v\in M^{\rm div}_\mathbb Q}(A_M(\sigma\cdot \mathfrak  v)+(\phi_{\sigma\cdot\mathcal F}-\phi_{\rm triv})(\sigma\cdot \mathfrak  v))\\
=&\inf_{ \mathfrak v\in M^{\rm div}_\mathbb Q}(A_M( \mathfrak v)+(\phi_{\mathcal F}-\phi_{\rm triv})( \mathfrak v))=L^{\rm NA}(\mathcal F)
\end{align*}
and we get the Proposition.

\end{proof}

To prove the equivariance of the minimizer of $H(\cdot)$, we need the uniqueness theorem. The uniqueness of the minimizer has been proved in \cite[Section 6]{Han-Li} provided $H(\cdot)$ has a special minimizer. This is the case when $M$ is smooth. In this case, it is proved in \cite[Section 3]{Chen-Wang-Sun} that the semistable limit $(\mathcal X,\Lambda)$ introduced on pp.2-3 already gives a special $\mathbb R$-test configuration $\mathcal F_{\min}$ which is a minimizer of $H(\cdot)$. Thus by \cite[Section 6]{Han-Li} $\mathcal F_{\min}$ is the unique minimizer of $H(\cdot)$. Later \cite[Theorem 1.2]{Blum-Liu-Xu-Zhuang} proved the uniqueness for general $\mathbb Q$-Fano varieties without the assumption of the existence of a special minimizer. It is proved by \cite[Theorem 1.2]{Blum-Liu-Xu-Zhuang} that on a $\mathbb Q$-Fano variety $H(\cdot)$ always admits a unique minimizer $\mathcal F_{\rm min}$ and $\mathcal F_{\min}$ must be a special $\mathbb R$-test configuration.

\begin{coro}\label{equ-minimizer}
Let $M$ be a $\mathbb Q$-Fano variety. Then the minimizer of $H(\cdot)$ is $\mathfrak G$-equivariant.
\end{coro}
\begin{proof}
Suppose that $\mathcal F_{\min}$ is a minimizer of $H(\cdot)$. By \eqref{H-inv-by-G}, for any $\sigma\in \mathfrak G$, $H(\sigma\cdot\mathcal F_{\min})=H(\mathcal F_{\min})$. By the uniqueness, $\sigma\cdot\mathcal F_{\min}=\mathcal F_{\min}$ for any $\sigma\in \mathfrak G$. Hence we get \eqref{group-eq-F} and $\mathcal F_{\min}$ is $\mathfrak G$-equivariant.
\end{proof}

\subsection{Spherical varieties}
In the following we overview the theory of spherical varieties. The origin goes back to \cite{Luna-Vust}. We use \cite{Kn91,Timashev-book} as main references.
\begin{defi} Let $\hat G$ be a connected, complex reductive group. A normal variety $M$ equipped with a $\hat G$-action is called a ($\hat G$-)spherical variety if there is a Borel subgroup $\hat B$ of $\hat G$ acts on $M$ with an open dense orbit.
\end{defi}
In particular, if a subgroup $\hat H\subset \hat G$ is called a \emph{spherical subgroup} if $\hat G/\hat H$ is a spherical variety (referred as a \emph{spherical homogenous space}). For an arbitrary spherical variety $M$, let $x_0$ be a point in the open dense $\hat B$-orbit. Set $\hat H={\rm Stab}_{\hat G}(x_0)$, the stabilizer of $x_0$ in $\hat G$. Then $\hat H$ is spherical and we call $(M,x_0)$ a \emph{spherical embedding} of $\hat G/\hat H$.


\subsubsection{The coloured data}
Let $\hat H$ be a spherical subgroup of $\hat G$ with respect to the Borel subgroup $\hat B$. The action of $\hat G$ on the function field $\mathbb C(\hat G/\hat H)$ of $\hat G/\hat H$ is given by
$$(g^*f)(x):=f(g^{-1}\cdot x),~\forall g\in\hat G, x\in \mathbb C(\hat G/\hat H)~\text{and}~f\in\mathbb C(\hat G/\hat H).$$
A function $f(\not=0)\in\mathbb C(\hat G/\hat H)$ is called \emph{$\hat B$-semi-invariant} if there is a character of $\hat B$, denote by $\chi$ so that $b^*f=\chi(b)f$ for any $b\in\hat B$. Note that there is an open dense $\hat B$-orbit in $\hat G/\hat H$. Two $\hat B$-semi-invariant functions associated to a same character can differ from each other only by multiplying a non-zero constant.

Let $\mathfrak M(\hat G/\hat H)$ be the lattice of $\hat B$-characters which admits associated $\hat B$-semi-invariant functions, and $\mathfrak N(\hat G/\hat H)={\rm Hom}_\mathbb Z(\mathfrak M(\hat G/\hat H),\mathbb Z)$ its $\mathbb Z$-dual. There is a map $\varrho$ which maps a valuation $\nu$ of $\mathbb C(\hat G/\hat H)$ to an element $\varrho(\nu)$ in $\mathfrak N_\mathbb Q(\hat G/\hat H)=\mathfrak N(\hat G/\hat H)\otimes_\mathbb Z\mathbb Q$ by
$$\varrho(\nu)(\chi)=\nu(f),$$
where $f$ is a $\hat B$-semi-invariant functions associated to $\chi$. Again, this is well-defined since $\hat G/\hat H$ is spherical with respect to $\hat B$. It is a fundamental result that $\varrho$ is injective on $\hat G$-invariant valuations and the image forms a convex cone $\mathcal V(\hat G/\hat H)$ in $\mathfrak N_\mathbb  Q(\hat G/\hat H)$, called the \emph{valuation cone} of $\hat G/\hat H$ (cf. \cite[Section 19]{Timashev-book}). Moreover, $\mathcal V(\hat G/\hat H)$ is a solid cosimplicial cone which is a (closed) fundamental chamber of a certain crystallographic reflection group, called the {\it little Weyl group} (denoted by $\bar W$; cf. \cite[Sections 22]{Timashev-book}). In fact, $\bar W$ is the Weyl group of the \emph{spherical root system} $\bar \Phi(\hat G/\hat H)$ of $\hat G/\hat H$ (cf. \cite[Section 30]{Timashev-book}).

\begin{exa}\label{grp-exa}
Let $G$ be a connected, complex reductive group. Let $B^+\subset G$ be a Borel subgroup of $G$ and $B^-$ be its opposite group. Take $\hat G=G\times G$, $\hat H={\rm diag}(G)$ and $\hat B=B^+\times B^-$. Then by the well-known Bruhat decomposition, $\hat H$ is a spherical subgroup of $\hat G$. Hence a $G$-compactification is a $\hat G$-spherical variety. By taking an involution on $\hat G$ $$\sigma(g_1,g_2)=(g_2,g_1)~\forall g_1, g_2\in G,$$ we see that $\hat H=\hat G^\sigma$. Thus $G\cong\hat G/\hat H$ is in addition a symmetric space.

Let us fix the maximal complex torus $T^\mathbb C=B^+\cap B^-$ of $G$. Denote by $\Phi$ the root system of $(G, T^\mathbb C)$ and $\Phi_+\subset\Phi$ the positive roots with respect to $B^+$. Let $W$ be the corresponding Weyl group. By a direct computation we can identify $\mathfrak M(\hat G/\hat H)$ with the lattice of weights $\mathfrak M(G)$ of $G$ via the anti-diagonal embedding,
\begin{align*}
\mathfrak M(G)&\to\mathfrak M(\hat G/\hat H)\\
\chi&\to(\chi,-\chi).
\end{align*}
Further calculation shows that under this identification the spherical root system $\bar \Phi(\hat G/\hat H)$ is identified with $\Phi$. Consequently, $\bar W\cong W$. Hence we identify $\mathcal V(\hat G/\hat H)$ with the anti-dominant (closed) Weyl chamber of $W$ in $\mathfrak a:=\mathfrak N_\mathbb R(G)$ (cf. \cite[Sections 4-8]{Timashev-Sbo}),
$$-\overline{\mathfrak a_+}:=\{y\in\mathfrak a|\alpha(y)\leq0,~\forall\alpha\in\Phi_+\}.$$
\end{exa}

A $\hat B$-stable prime divisors in $\hat G/\hat H$ is called a \emph{colour}. Denote by $\mathcal D_{\hat B}(\hat G/\hat H)$ the set of colours. A colour $D\in\mathcal D_{\hat B}(\hat G/\hat H)$ also defines a valuation on $\hat G/\hat H$. However, the restriction of $\varrho$ on $\mathcal D_{\hat B}(\hat G/\hat H)$ is usually non-injective.

\begin{exa}
Let $\hat H\supset \hat B$ be a parabolic subgroup of $\hat G$. Then $\varrho$ vanishes on $\mathcal D_{\hat B}(\hat G/\hat H)$.
\end{exa}

It is a fundamental result that the spherical embeddings of a given $\hat G/\hat H$ are classified by combinatorial data called the \emph{coloured fan} \cite{Luna-Vust},
\begin{defi} Let $\hat H$ be a spherical subgroup of $\hat G$, $\mathcal D_{\hat B}(\hat G/\hat H), \mathcal V(\hat G/\hat H)$ be the set of colours and the valuation cone, respectively.
\begin{itemize}
\item A {coloured cone} is a pair $(\mathcal C,\mathcal R)$, where $\mathcal R \subset  \mathcal \mathcal D_{\hat B}(\hat G/\hat H), O\not\in\varrho(\mathcal R)$ and $\mathcal C \subset \mathfrak N_\mathbb Q(\hat G/\hat H)$  is a strictly convex cone generated by $\varrho(\mathcal R)$ and finitely many elements of $\mathcal V(\hat G/\hat H)$ such
that the intersection of the relative interior of $\mathcal C$ with $\mathcal V(\hat G/\hat H)$ is non-empty;
\item Given two coloured cones $(\mathcal C,\mathcal R)$ and $(\mathcal C',\mathcal R')$, We say that a coloured cone $(\mathcal C',\mathcal R')$ is a face
of another coloured cone $(\mathcal C,\mathcal R)$ if $\mathcal C'$ is a face of $\mathcal C$ and $\mathcal R' = \mathcal R \cap \varrho^{-1}(\mathcal C')$;
\item A {coloured fan} is a collection $\mathcal F$ of finitely many coloured cones such that the face of any its coloured cone is still in it, and any $v \in  \mathcal V(\hat G/\hat H)$ is contained in the relative interior of at most one of its cones.
\end{itemize}
\end{defi}

Now we state the classification theorem of spherical embeddings (cf. \cite[Theorem 3.3]{Kn91}),
\begin{theo}
There is a bijection $(M,x_0) \to\mathcal F_M$ between spherical embeddings of $\hat G/\hat H$ up to $\hat G$-equivariant isomorphism and coloured fans. There is a
bijection $\mathcal Y \to (\mathcal C_{\mathcal Y},\mathcal R_{\mathcal Y})$ between the $\hat G$-orbits in $M$ and the coloured cones in $\mathcal F_M$.
An orbit $\mathcal Y$ is in the closure of another orbit $\mathcal Z$ in $M$ if and only if the coloured cone $(\mathcal{C_Z},\mathcal{R_Z})$ is a face of $(\mathcal{C_Y},\mathcal{R_Y})$.
\end{theo}


\subsubsection{Line bundles and polytopes}
Let $M$ be a complete spherical variety, which is a spherical embedding of some $\hat G/\hat H$. Let $L$ be a $\hat G$-linearlized line bundle on $M$. In the following we will associated to $(M,L)$ several polytopes, which encode the geometric structure of $M$.
\subsubsection*{Moment polytope of a line bundle}
Let $(M,L)$ be a polarized spherical variety. Then for any $k\in\mathbb N$ we can decompose $H^0(M,L^k)$ as direct sum of irreducible $\hat G$-representations,
\begin{align}\label{H0kL}
H^0(M,L^k)=\oplus_{\hat\lambda\in P_{+,k}}\hat V_{\hat\lambda},
\end{align}
where $P_{+,k}$ is a finite set of $\hat B$-weights and $\hat V_{\hat\lambda}$ the irreducible representation of $\hat G$ with highest weight $\hat\lambda$. Set
$$P_+:=\overline{\cup_{k\in\mathbb N}(\frac1kP_{+,k})}.$$
Then $P_+$ is indeed a polytope in $\mathfrak M_\mathbb R(\hat G/\hat H)$. We call it the \emph{moment polytope of $(M,L)$}. Clearly, the moment polytope of $(M,L^k)$ is $k$-times the moment polytope of $(M,L)$ for any $k\in\mathbb N_+$.

\subsubsection*{Polytope of a divisor}
Denote by $\mathcal I_{\hat G}(M)=\{D_A|A=1,...,d_0\}$ the set of $\hat G$-invariant prime divisors in $M$. Then any $D_A\in\mathcal I_{\hat G}(M)$ corresponds to a $1$-dimensional cone $(\mathcal C_A,\emptyset)\in\mathcal F_M$. Denote by $u_A$ the prime generator of $\mathcal C_A$. Recall that $\mathcal D_{\hat B}(\hat G/\hat H)$ is the set of colours, which are $\hat B$-stable but not $\hat G$-stable in $M$. Any $\hat B$-stable $\mathbb Q$-Weil divisor can be written as
\begin{align}\label{weil-div}
{\mathfrak d}=\sum_{A=1}^{d_0}\lambda_AD_A+\sum_{D\in\mathcal D_{\hat B}(\hat G/\hat H)}\lambda_DD
\end{align}
for some $\lambda_A,\lambda_D\in\mathbb Q$. Set
$$\mathcal D_M:=\cup\{\mathcal R\subset\mathcal D_{\hat B}(\hat G/\hat H)|~\exists(\mathcal C,\mathcal R)\in\mathcal F_M\}$$
By \cite[Proposition 3.1]{Brion89}, ${\mathfrak d}$ is further a $\mathbb Q$-Cartier divisor if and only if there is a rational piecewise linear function $l_{\mathfrak d}(\cdot)$ on $\mathcal F_M$ such that
$$\lambda_A=l_{\mathfrak d}(u_A),~A=1,...,d_0~\text{and}~\lambda_D=l_{\mathfrak d}(\varrho(D)),~\forall D\in\mathcal D_M.$$
It is further proved in \cite[Section 3]{Brion89} that when ${\mathfrak d}$ is ample $l_{\mathfrak d}(-x):\mathfrak N_\mathbb R(\hat G/\hat H)\to\mathbb R$ is the support function of some convex polytope $\Delta_{\mathfrak d}$. We call the $\Delta_{\mathfrak d}$ the \emph{polytope of ${\mathfrak d}$}.

Suppose that $s$ is a $\hat B$-semi-invariant section of $L$ with respect to a character $\chi$. Let ${\mathfrak d}$ be the divisor of $s$. We have
\begin{prop}\label{polytope-prop}(\cite[Proposition 3.3]{Brion89})
The two polytopes $P_+$ and $\Delta_{\mathfrak d}$ are related by $$P_+=\chi+\Delta_{\mathfrak d}.$$
\end{prop}

When $L=K_M^{-1}$, there is a divisor ${\mathfrak d}$ of $L$ in form of \eqref{weil-div},
\begin{align*}
{\mathfrak d}=\sum_{A=1}^{d_0}D_A+\sum_{D\in\mathcal D_{\hat B}(\hat G/\hat H)}n_DD,
\end{align*}
where $n_D$ are explicitly obtained in \cite{Ga-Ho}. In particular, $n_D\equiv2$ for group compactifications. We associated to ${\mathfrak d}$ one more polytope as the convex hull
\begin{align}\label{Delta-d-dual}
\Delta_{\mathfrak d}^*={\rm Conv}(\{u_A|A=1,...,d_0\}\cup\{\frac{\varrho(D)}{n_D}|~D\in\mathcal D_{\hat B}(\hat G/\hat H)\}).
\end{align}

In particular, the coloured fan $\mathcal F_M$ can be recovered from $\Delta^*_{\mathfrak d}$ by taking all coloured cones $({\rm Cone}(F),\varrho^{-1}(F))$ for all faces $F$ of $\Delta_{\mathfrak d}^*$ such that ${\rm RelInt}({\rm Cone}(F))\cap\mathcal V(\hat G/\hat H)\not=\emptyset$. We will use this construction in Section 3.

\subsubsection{The Okounkov body of a spherical variety}
For each dominant weight $\hat\lambda$ of $\hat G$, there is a Gel'fand-Tsetlin polytope $\Delta(\hat\lambda)$ which has the same dimension with the maximal unipotent subgroup $\hat N_u$ of $\hat G$ (cf. \cite{Kirichenko-Smirnov-Timorin}). It is known that
\begin{align}\label{okounkov-body-fibre}
\dim(V_{\hat\lambda})=\text{number of integral points in $\Delta(\hat\lambda)$}.
\end{align}

Let $M$ be a $\hat G$-spherical variety. It is proved in \cite{Okounkov} that the Okounkov body $\Delta$ of $M$ is given by the convex hull
\begin{align}\label{okounkov-body}
\Delta:={\rm Conv}\left(\cup_{k\in\mathbb N_+}\cup_{\hat\lambda\in\overline{P_+}\cup\frac1k\mathfrak M(\hat G/\hat H)}(\hat\lambda,\frac1k\Delta(k\hat\lambda))\right)\subset\mathfrak M_\mathbb R(\hat G/\hat H)\oplus\mathbb R^{\dim(\hat N_u)}.
\end{align}
Note that the Gel'fand-Tsetlin polytope $\Delta(\hat\lambda)$ is linear in $\hat\lambda$. Thus $\Delta$ is a convex polytope in $\mathfrak M_\mathbb R(\hat G/\hat H)\oplus\mathbb R^{\dim(\hat N_u)}.$ It is a fibration over $\overline{P_+}$ so that the fibre at each $\hat\lambda\in\overline{P_+}\cup\frac1k\mathfrak M(\hat G/\hat H)$ is $\frac1k\Delta(k\hat\lambda)$.

\subsection*{Notations} Now we fix the notations in the following sections. We denote by
\begin{itemize}
\item $K$-a connected, compact Lie group;
\item $G=K^\mathbb C$-the complexification of $K$, which is a complex, connected reductive Lie group;
\item $J$-the complex structure of $G$;
\item $T$-a fixed maximal torus of $K$ and $T^\mathbb C$ its complexification;
\item $B^+$-a chosen positive Borel group of $G$ containing $T^\mathbb C$ and $B^-$ the opposite one;
\item $\mathfrak a:=J\mathfrak t$-the non-compact part of $\mathfrak t^\mathbb C$ and $\mathfrak a^*$ the dual of $\mathfrak a$;
\item $\Phi$-the root system with respect to $G$ and $T^\mathbb C$;
\item $W$-the Weyl group with respect to $G$ and $T^\mathbb C$;
\item $\Phi_+$-a chosen system of positive roots in $\Phi$ determined by $B^+$ and $\Phi_{+,s}\subset\Phi_+$ the simple roots;
\item $\mathfrak a_+$ and $\mathfrak a^*_+$-the dominant Weyl chamber with respective to $\Phi_+$ in $\mathfrak a$ and $\mathfrak a^*$, respectively;
\item $\langle\cdot,\cdot\rangle$-a fixed $W$-invariant inner product on $\mathfrak a$;
\item For any dominant weight $\lambda$ of $G$, denote by $V_\lambda$ the irreducible representation of $G$ with highest weight $\lambda$ and $v_\lambda$ the highest weight vector. Also denote by $V_\lambda^*$ the dual representation of $V_\lambda$. Then $V_\lambda^*$ has a vector $v^*_{-\lambda}$ of lowest weight $-\lambda$;
\item $\text{Ad}_\sigma(\cdot):=\sigma(\cdot)\sigma^{-1}$-the conjugate of some subgroup or Lie algebra by some element $\sigma$.
\item $\hat G=G\times G$, $\hat T=T\times T$ and $\hat B^+=B^-\times B^+$;
\item $\hat U^+\subset \hat B^+$-the maximal unipotent subgroup in $\hat B^+$;
\item $\hat \Phi, \hat\Phi_+$-the roots and positive roots with respect to $\hat G$ and $\hat B^+$, respectively;
\item $\mathcal V(\cdot)$-the valuation cone of some spherical homogenous space; 
\item $\mathfrak M(\cdot)$-certain lattice of weights and $\mathfrak N(\cdot)={\rm Hom}_\mathbb Z(\mathfrak M(\cdot),\mathbb Z)$;
\item $\pi_\nu$-the projection from $\mathfrak N(\hat T)$ to $\mathfrak N_\mathbb R(\hat G/\hat H)$.

\end{itemize}

\section{Equivariant normal $\mathbb Z$-test configurations}
In this section we overview useful results on the equivariant normal $\mathbb Z$-test configurations of a group compactification. Then we compute some combinatorial data of the central fibre.
\subsection{The classification results}
The equivariant $\mathbb Z$-test configurations of general spherical varieties are studied in \cite{Del3}. The following Proposition is a special case of \cite[Theorem 3.30]{Del3} for group compactifications.
\begin{prop}\label{constrc-test-congifg}
Let $M$ be a $\mathbb Q$-Fano $G$-compactification. Then for any $\mathbf{\Lambda}\in\mathfrak N(\hat T)\cap\pi_\nu^{-1}(\overline{\mathfrak a_+})$ and $m\in\mathbb N_+$, there is a $\hat G$-equivariant normal test configuration $(\mathcal X,\mathcal L)$ of $(M,K_M^{-1})$ with irreducible central fibre $\mathcal X_0$. Moreover, the central fibre $\mathcal X_0$ of $\mathcal X$ is a $\hat G$-equivariant embedding of $\hat G/ H_0$ for some spherical subgroup $H_0\subset \hat G$ and the $\mathbb C^*$-action on $\mathcal X_0$ is given by
\begin{align*}
e^z\cdot \hat gH_0=\hat g\mathbf{\Lambda}(e^{\frac zm})H_0,~\forall e^z\in\mathbb C^*.
\end{align*}
In addition, two vectors $(\mathbf{\Lambda},m)$ and $(\mathbf{\Lambda'},m)$ generate the same test configuration if $\pi_\nu(\mathbf{\Lambda})=\pi_\nu(\mathbf{\Lambda'})$.
\end{prop}
Indeed, up to multiplying $(\mathbf{\Lambda},m)$ by a sufficiently divisible integer, we can do the above construction for any $\mathbf{\Lambda}\in\mathfrak N_\mathbb Q(\hat T)\cap\pi_\nu^{-1}(\overline{\mathfrak a_+})$ and $m\in\mathbb Q_+$.

We briefly recall the construction of \cite{Del3}. Let ${\mathfrak d}$ be a $\hat B$-invariant divisor of $L$ and $\Delta_{\mathfrak d}^*$ the corresponding polytope given by \eqref{Delta-d-dual}.
The coloured cone $\mathcal F_{\mathcal X}$ of $\mathcal X$ consists of all cones of the following three types, which has non-empty intersection with the relative interior of $\mathcal V(\hat G/{\rm diag}(G))$,
\begin{align}\label{test-conf-fan}
&(\text{Cone}(F),\varrho^{-1}(\text{Cone}(F)));\notag\\
&(\text{Cone}(F\cup\{(\mathbf{0},1)\}),\varrho^{-1}(\text{Cone}(F\cup\{(\mathbf{0},1)\})));\notag\\
&(\text{Cone}(F\cup\{(-\mathbf{\Lambda},-m)\}),\varrho^{-1}(\text{Cone}(F\cup\{(-\mathbf{\Lambda},-m)\}))).
\end{align}
where $F$ runs over all faces of $\Delta_{\mathfrak d}^*$. Then $\mathcal X$ is a complete spherical embedding of $(\hat G\times \mathbb C^*)/({\rm diag}(G)\times \{e\})$ and there is a $\hat G\times\mathbb C^*$-equivariant surjective map $\pi_\mathcal X$ of $\mathcal X$ to $\mathbb{CP}^1$. Clearly $\mathcal X_0:=\pi_\mathcal X^{-1}(0)$ corresponds to the one-dimensional coloured cone $(-(\mathbf{\Lambda},m),\emptyset)$. 
The line bundle $\mathcal L$ is defined by the $\hat B$-invariant divisor
\begin{align}\label{TC-divisor}
\hat{\mathfrak d}=m_0(\sum_{A=1}^{d_0}\overline{D_A\times \mathbb C^*}+\sum_{D\in\mathcal D_{\hat B}(\hat G/\hat H)}n_D\overline{D\times \mathbb C^*}+(C_0-2\mathbf{\Lambda}_0(\rho))\mathcal X_{0,{\rm red}}),
\end{align}
where $m_0,C_0\in\mathbb N_+$ are sufficiently divisible constants so that $\hat{\mathfrak d}$ is an ample Cartier divisor. The number $m_0$ is called the \emph{exponent} of $(\mathcal X,\mathcal L)$. It is also proved by \cite[3.24]{Del3} that $H_0$ is a spherical subgroup of $\hat G$.

We will also use the following inverse of Proposition \ref{constrc-test-congifg} later.
\begin{prop}\label{parameter-of-test-config}
For any $\hat G$-equivariant normal test configuration $\mathcal X$ of $M$ with irreducible central fibre $\mathcal X_0$, there is an integral vector $(\Lambda,0,m)\in(\mathfrak a_+\cap\mathfrak N(T))\oplus\mathfrak N(T)\oplus\mathbb Z$ such that $\mathcal X$ is constructed from $(\Lambda,0,m)$ by using Proposition \ref{constrc-test-congifg}.
\end{prop}

\begin{proof}
By gluing $\mathcal X$ with a trivial family
$$M\times \mathbb C\to M$$
along $\mathbb C^*\subset\mathbb C$, we get an $\hat G$-equivariant family $\bar{\mathcal X}$
$$\bar\pi:\bar{\mathcal X}\to\mathbb {CP}^1$$
over $\mathbb {CP}^1$ such that $\bar\pi^{-1}(0)=\mathcal X_0$ and $\bar\pi^{-1}(t)=\mathcal X$ for $t=\infty$ and any $t\not=0$ in $\mathbb C$.

Note that the total space $\bar{\mathcal X}$ is a $(G\times \mathbb C^*)$-compactification. Consider the coloured fan $\mathcal F_{\bar{\mathcal X}}$ of $\bar{\mathcal X}$. Since the central fibre $\mathcal X_0$ is irreducible, it is a single $\hat G\times\mathbb C^*$-invariant divisor, which is associated to a $1$-dimensional cone in $\mathcal F_{\bar{\mathcal X}}$. Let $(-\Lambda_1,-\Lambda_2,-m)\in\mathfrak N(\hat T)\times\mathbb Z$ be the generator of this cone. Then $(\Lambda_1,\Lambda_2)\in\mathfrak N(\hat T)\cap\pi_\nu^{-1}(\overline{\mathfrak a_+})$. Take $\Lambda=\Lambda_1-\Lambda_2$, it is direct to check that $\mathcal X$ can be constructed from $(\Lambda,0,m)$ by using Proposition \ref{constrc-test-congifg}.

\end{proof}

\subsection{Combinatorial data of the central fibre}
In the following we will determine some combinatorial data of $\mathcal X_0$. Let $\Phi_{+,s}=\{\alpha_1,...,\alpha_r\}$ be the simple roots in $\Phi_+$. Then each $\alpha_i\in\Phi_{+,s}$ defines a Weyl wall $W_{\alpha_i}$ of the dominant Weyl chamber $\mathfrak a_+$ of $\mathfrak a$. As $\mathbf{\Lambda}\in\pi_\nu^{-1}(\overline{\mathfrak a_+})$, we can assume that $\mathbf{\Lambda}\in\pi_\nu^{-1}(W_{\alpha_i})$ for $i=1,...,i_0$ but away from other Weyl walls, or equivalently,
\begin{align}\label{wall-Lambda}
\alpha_i(\Lambda_1-\Lambda_2)=0,~i=1,...,i_0,
\end{align}
for simple roots $\alpha_{1},...,\alpha_{i_0}\in\Phi_{+,s}$ and
\begin{align}\label{no-wall-Lambda}
\alpha_i(\Lambda_1-\Lambda_2)>0,~i=i_0+1,...,r,
\end{align}
for the remaining simple roots $\alpha_{i_0},...,\alpha_r$. Also, let $\alpha_{r+1},...,\alpha_{s_1}$ be positive roots in $\Phi_+\setminus\Phi_{+,s}$ which can be written as linear combination of $\alpha_{1},...,\alpha_{i_0}$. Denote by $\alpha_{s_1+1},...,\alpha_{\frac{n-r}2}$ the remaining positive roots in $\Phi_+\setminus\Phi_{+,s}$.

As mentioned before, $\mathcal X_0$ is an equivariant embedding of some spherical homogenous space $\hat G/H_0$. For our latter use, we compute the data of $H_0$,
\begin{prop}\label{H0}
Suppose that $\mathbf{\Lambda}=(\Lambda_1,\Lambda_2)\in\mathfrak N(T)\oplus\mathfrak N(T)\cong \mathfrak N(\hat T)$ satisfies \eqref{wall-Lambda}-\eqref{no-wall-Lambda}. Then the central fibre $\mathcal X_0$ is a $\hat G$-equivariant compactification of $\hat G/H_0$, where $H_0$ is a subgroup of $\hat G$ with Lie algebra
\begin{equation}\label{Lie-H0-app}
\begin{aligned}
\mathfrak h_0=&{\rm diag}((\Lambda_2-\Lambda_1)^\perp)\oplus\mathbb C(\Lambda_1,\Lambda_2)\\
&\oplus\oplus_{i=1,...,i_0;r+1,..,s_1}(\mathbb C(X_{\alpha_i},X_{\alpha_i})\oplus\mathbb C(X_{-\alpha_i},X_{-\alpha_i}))\\
&\oplus\oplus_{j=i_0+1,...,r;s_1+1,..,\frac{n-r}2}(\mathbb C(0,X_{\alpha_j})\oplus\mathbb C(X_{-\alpha_j},0)).
\end{aligned}
\end{equation}
Here $(\Lambda_2-\Lambda_1)^\perp$ is the orthogonal complement of $\mathbb C(\Lambda_2-\Lambda_1)$ in $\mathfrak a$.
\end{prop}
\begin{proof}
By \cite[Proposition 3.23]{Del3}, we can find an $x_0$ in $\mathcal X_0$ whose $\hat G\times \mathbb C^*$-orbit is open dense in $\mathcal X_0$ and is isomorphic to $(\hat G\times \mathbb C^*)/\hat H_0$ for some spherical $\hat H_0\subset(\hat G\times\mathbb C^*)$. Also, $x_0$ can be realized as
\begin{align}\label{base-point-X0}
x_0=\lim_{\mathbb C^*\ni t\to0}(\mathbf{\Lambda}(t),t^m)\hat x_0
\end{align}
for some $\hat x_0$ in the open dense $\hat G\times \mathbb C^*$-orbit of $\mathcal X$.

We first compute $\hat H_0$. Consider the base point $\hat x_0$ of the open dense orbit of $\mathcal X$ in \eqref{base-point-X0}. Its stabilizer in $\hat G\times \mathbb C^*$ is $$\hat H={\rm diag}(G)\times\{e\},$$ whose Lie algebra is spanned by
\begin{align*}
&(X,X,0),~X\in\mathfrak t;\notag\\
&(X_\alpha,X_\alpha,0),~\alpha\in\Phi_+;\notag\\
&(X_{-\alpha},X_{-\alpha,0}),~\alpha\in\Phi_+.
\end{align*}
Recall that for each $t\in\mathbb C$, $(\mathbf{\Lambda}(e^t),e^{mt})\hat x_0$ has stabilizer $\text{Ad}_{(\mathbf{\Lambda}(e^t),e^{mt})}H$, whose Lie algebra is spanned by
\begin{align}
\text{Ad}_{(\mathbf{\Lambda}(e^t),e^{mt})}=&(X,X,0),~X\in\mathfrak t;\label{conj-torus}\\
\text{Ad}_{(\mathbf{\Lambda}(e^t),e^{mt})}(X_\alpha,X_\alpha,0)=&(e^{\alpha(\Lambda_1)t}X_\alpha,e^{\alpha(\Lambda_2)t}X_\alpha,0),~\alpha\in\Phi_+;\notag\\
\text{Ad}_{(\mathbf{\Lambda}(e^t),e^{mt})}(X_{-\alpha},X_{-\alpha},0)=&(e^{-\alpha(\Lambda_1)t}X_{-\alpha},e^{-\alpha(\Lambda_2)t}X_{-\alpha},0),~\alpha\in\Phi_+.
\end{align}
By \eqref{wall-Lambda} and the above relations, we have
\begin{align}\label{conj-X-alpha}
\text{Ad}_{(\mathbf{\Lambda}(e^t),e^{mt})}(X_\alpha,X_\alpha,0)=(X_\alpha,X_\alpha,0)
\end{align}
and
\begin{align}\label{conj-X-minus-alpha}
\text{Ad}_{(\mathbf{\Lambda}(e^t),e^{mt})}(X_{-\alpha},X_{-\alpha},0)=(X_{-\alpha},X_{-\alpha},0)
\end{align}
for all $\alpha\in\{\alpha_1,...,\alpha_{i_0},\alpha_{r+1},...,\alpha_{s_1}\}$.

On the other hand, by \eqref{no-wall-Lambda}, as $e^t\to0$ we have
$$e^{-\alpha(\Lambda_2)t}\text{Ad}_{(\mathbf{\Lambda}(e^t),e^{mt})}(X_\alpha,X_\alpha,0)=(e^{\alpha(\Lambda_1-\Lambda_2)t}X_\alpha,X_\alpha,0)\to(0,X_{\alpha},0)$$
and
$$e^{\alpha(\Lambda_1)t}\text{Ad}_{(\mathbf{\Lambda}(e^t),e^{mt})}(X_{-\alpha},X_{-\alpha},0)=(X_{-\alpha},e^{\alpha(\Lambda_1-\Lambda_2)t}X_{-\alpha},0)\to(X_{-\alpha},0,0)$$
for all $\alpha\in\{\alpha_{i_0+1},...,\alpha_r,\alpha_{s_1+1},...,\alpha_{\frac{n-r}2}\}$.

It is direct to see that $(\mathbf{\Lambda},m)\in\hat{\mathfrak h}_0$. Hence the Lie algebra $\hat{\mathfrak h}_0$ of $\hat H_0$ is
\begin{align*}
\hat{\mathfrak h}_0=&({\rm diag}(\mathfrak t)\times\{0\})\oplus\mathbb C(\Lambda_1,\Lambda_2,m)\notag\\
&\oplus\left(\oplus_{i=1,...,i_0;r+1,...,s_1}\mathbb C(X_{\alpha_i},X_{\alpha_i},0)\oplus\mathbb C(X_{-\alpha_i},X_{-\alpha_i},0)\right)\notag\\
&\oplus\left(\oplus_{i=i_0+1,...,r;s_1+1,...,\frac{n-r}2}\mathbb C(0,X_{\alpha_i},0)\oplus\mathbb C(X_{-\alpha_i},0,0)\right),
\end{align*}
which is understood as a Lie sub-algebra in $\mathfrak g\oplus\mathfrak g\oplus\mathbb C$. \eqref{Lie-H0-app} then follows directly from the above relation.
\end{proof}

By the above Proposition we can show that $H_0$ is even horosymmetric in the sense of \cite[Definition 2.1]{Del4}.
\begin{coro}\label{H0-appendix}
Under the assumption of Proposition \ref{H0}, the homogenous space $\hat G/H_0$ is horosymmetric. Its anticanonical line bundle has isotropic character
\begin{align}\label{chi-app}
\chi=\sum_{j=i_0+1,...,r;s_1+1,...,\frac{n-r}2}({\alpha_j},{-\alpha_j}).
\end{align}
\end{coro}

\begin{proof}
To see that $H_0$ is horosymmetry, consider the following parabolic subgroup $\hat P\subset\hat G$ with Lie algrbra
\begin{align*}
\hat{\mathfrak p}=(\mathfrak t\oplus\mathfrak t)
&\oplus\oplus_{i=1,...,i_0;r+1,...,s_1}(\mathbb C(X_{\pm\alpha_i},0)\oplus\mathbb C(0,X_{\pm\alpha_i}))\\
&\oplus\oplus_{j=i_0+1,...,r;s_1+1,...,\frac{n-r}2}(\mathbb C(X_{-\alpha_j},0)\oplus\mathbb C(0,X_{\alpha_j})),
\end{align*}
and Levi group $\hat L=L\times L$ in $\hat P$, where $L$ has Lie algebra
$$\mathfrak l=\mathfrak t\oplus\oplus_{i=1,...,i_0;r+1,...,s_1}(\mathbb CX_{\pm\alpha_i}\oplus\mathbb CX_{-\alpha_i}).$$
By Proposition \ref{H0}, $H_0\subset P$ and the unipotent radical
\begin{align}\label{hatPu-app}
\hat P^u\subset H_0.
\end{align}
Define an involution $\Theta$ on $\hat{\mathfrak l}$ whose eigenspace of $+1$ is
$${\rm diag}((\Lambda_2-\Lambda_1)^\perp)\oplus\mathbb C(\Lambda_1,\Lambda_2)\oplus\oplus_i\mathbb C(X_{\pm\alpha_i},X_{\pm\alpha_i}),$$
and  eigenspace of $-1$ is
$${\rm antidiag}((\Lambda_2-\Lambda_1)^\perp)\oplus\mathbb C(\Lambda_2,-\Lambda_1)\oplus\oplus_i\mathbb C(X_{\pm\alpha_i},-X_{\pm\alpha_i}),$$
where ${\rm antidiag}(V)$ denotes the anti-diagnal embedding of $V$ in $V\times V$.

Since all $\alpha_i$'s are in $(\Lambda_2-\Lambda_1)^\perp$, it is not hard to check that $\Theta$ is a morphism of the Lie algebra $\hat{\mathfrak l}$. Hence it defines a complex involution $\Theta$ on $\hat L$. It is direct to check that the neutral component of the fixed points
$$(\hat L)^{\Theta}=\hat L\cap H_0\subset H_0.$$
Combing with \eqref{hatPu-app}, we see that $H_0$ is horosymmetry. Relation \eqref{chi-app} then follows from \cite[Example 3.1]{Del4}.
\end{proof}

\subsubsection{The equivariant automorphism}
Now we compute $\text{Aut}_{\hat G}(\mathcal X_0)$, the group of $\hat G$-equivariant automorphisms of $\mathcal X_0$. Fix a $W$-invariant inner product $\langle\cdot,\cdot\rangle$ on $\mathfrak a$ which extends the Killing form on $\mathfrak a\cap\mathfrak{[g,g]}$. Take ${\mathfrak a}_1=\mathfrak a\cap(\cap_{i=1,...,i_0}\ker(\alpha_i))$ and ${\mathfrak a}_2$ the orthogonal complement of $\mathfrak a_1$ in ${\mathfrak a}$. Let $\hat A_1,\hat A_2$ be two toruses of $\hat G$ defined by
$$\hat A_1=\exp(\mathbb C{\mathfrak a}_1\oplus \mathbb C{\mathfrak a}_1),~\hat A_2=\exp(\mathbb C{\mathfrak a}_2\oplus\mathbb C {\mathfrak a}_2).$$
We conclude from \eqref{conj-torus}-\eqref{conj-X-minus-alpha} that the centralizer
$$C_{\hat B}(H_0)\cap N_{\hat G}(H_0)= \hat A_1.$$
By \cite[Proposition 3.21 and 3.24]{Del3}, for the adapted Levi group $\hat B$,
$$N_{\hat G}(H_0)=H_0(C_{\hat B}(H_0)\cap N_{\hat G}(H_0))=H_0(\hat A_1)=H_0(\exp(\mathfrak a_1)\times\{e\}).$$
By \cite[Proposition 1.8]{Timashev-book}, the group of $\hat G$-equivariant automorphisms of $\mathcal X_0$, $${\rm Aut}_{\hat G}(\mathcal X_0)\cong{\rm Aut}_{\hat G}(\hat G/ H_0)\cong N_{\hat G}(H_0)/ H_0.$$
Thus,we have
\begin{lem}\label{equi-aut}
Let $A_1=\exp({\mathfrak a}_1)\subset T.$ Then
$${\rm Aut}_{\hat G}(\mathcal X_0)\cong A_1.$$
\end{lem}
The Lemma will be used for computing the valuation cone of $\hat G/H_0$.

\subsubsection{The valuation cone of $\hat G/H_0$}
In this section we compute the valuation cone $\mathcal V(\hat G/H_0)$. We will adopt the formal curve method in \cite[Section 24]{Timashev-book}. Suppose that it holds \eqref{wall-Lambda}-\eqref{no-wall-Lambda}. Set
$$\hat U_1=\exp(\oplus_{i=1,...,i_0;r+1,...,s_1})(\mathbb C(0,X_{\alpha_i})\oplus\mathbb C(X_{-\alpha_i},0))$$
and
$$\hat U_2=\exp(\oplus_{j=i_0+1,...,r;s_1+1,...,\frac{n-r}2})(\mathbb C(0,X_{\alpha_i})\oplus\mathbb C(X_{-\alpha_i},0)).$$
Then
\begin{align}\label{decompose-U}
\hat U^+=\hat U_1\cdot\hat U_2.
\end{align}

Also, define
\begin{align*}
\mathfrak l=&\mathfrak t^\mathbb C\oplus\left(\oplus_{i=1,...,i_0,r+1,...,s_1}(\mathbb CX_{\alpha_i}\oplus\mathbb CX_{-\alpha_i})\right)\\
L=&\exp(\mathfrak l).
\end{align*}
Then $L$ is a reductive subgroup of $G$, $T$ is a maximal compact torus of $L$ and $\Phi_{L}=\{\pm\alpha_i|i=1,...,i_0,r+1,...,s_1\}$ is its root system. Moreover, $\Phi_{L,+}=\Phi_L\cap\Phi_+$ and $\Phi_{L,+,s}=\Phi_L\cap\Phi_{+,s}$ are the positive and the simple roots, respectively. Set
\begin{align}\label{hat-L-levi}
\hat L=L\times L.
\end{align}

To apply the formal curve method, we need the following
\begin{lem}\label{formal-curve-decomposition-cartan}
A formal curve $\hat G((t))$ in $\hat G\times \mathbb C^*$ can be decomposed as
\begin{align}\label{formal-curve-decomposition-cartan-eq-1}
\hat G((t))=\hat G[[t]]\cdot\hat A((t))\cdot\hat L[[t]]\cdot \hat U_2((t)).
\end{align}
Consequently,
\begin{align}\label{formal-curve-decomposition-cartan-eq}
\hat G((t))H_0=\hat G[[t]]\cdot\hat A((t))\cdot\hat L[[t]]H_0.
\end{align}
\end{lem}
\begin{proof}
By the Iwasawa decomposition in \cite[Section 24]{Timashev-book} and \eqref{decompose-U},
\begin{align}
\hat G((t))=&\hat G[[t]]\cdot\hat A((t))\cdot \hat U^+((t))\notag\\
=&\hat G[[t]]\cdot\hat A_2((t))\cdot\hat A_1((t))\cdot \hat U_1((t))\cdot \hat U_2((t))\notag\\
=&\hat G[[t]]\cdot\hat A_2((t))\cdot\hat U_1((t))\cdot\hat A_1((t))\cdot \hat U_2((t)),\label{decom-01}
\end{align}
where in the last line we use the fact that $\hat A_1$ commutes with $\hat U_1$. Combining with the fact that $\hat U_2\subset H_0$, we have
\begin{align*}
\hat G((t))H_0=\hat G[[t]]\cdot\hat A_2((t))\cdot\hat U_1((t))\cdot\hat A_1((t))H_0.
\end{align*}

Set $\hat L_{ss}:=[\hat L,\hat L]$. By the last line of \eqref{decom-01}, we can rewrite
\begin{align*}
\hat G((t))=&\hat G[[t]]\cdot\hat L_{ss}((t))\cdot\hat A_1((t))\cdot \hat U_2((t))\notag\\
=&\hat G[[t]]\cdot\hat A_2((t))\cdot\hat L_{ss}[[t]]\cdot\hat A_1((t))\cdot \hat U_2((t)).
\end{align*}
Here we used the Cartan decomposition in \cite[Section 24]{Timashev-book} for $\hat L_{ss}((t))$ in the last line. Since $\hat A_1$ commutes with $\hat L_{ss}$, we get \eqref{formal-curve-decomposition-cartan-eq-1}. Note that by Proposition \ref{H0}, $\hat U_2\subset H_0$. We get \eqref{formal-curve-decomposition-cartan-eq}.
\end{proof}

Now we can compute the valuation cone
\begin{prop}\label{valuation-cone} Under the assumption of Proposition \ref{H0}, the valuation cone $\mathcal V(\hat G/ H_0)$ of $\hat G/ H_0$ can be identified with the anti-dominant Weyl chamber with respective to $\Phi_{L,+}$,
\begin{align}\label{val-cone}
\{x\in\mathfrak a|\alpha(x)\leq0,\alpha\in\Phi_{L,+}\}=\cap_{i=1}^{i_0}\{x\in\mathfrak a|\alpha_i(x)\leq0\}.
\end{align}
\end{prop}

\begin{proof}
We use the arguments of \cite[Section 24]{Timashev-book}. By Lemma \ref{formal-curve-decomposition-cartan}, every $v\in\hat{\mathfrak a}$ is proportional to a punctured curve in $\hat G/H_0$. It suffices to compute the order
$$\nu(f((g_1,g_2)\gamma(t)H_0)),~\text{as}~t\to0,$$
for a rational function $f$ in $\hat G/\hat H_0$ and a generic $(g_1,g_2)\in\hat G$. Decompose $v=(v_1,v_2)$ such that
$$v_1\in\hat{\mathfrak a}_1,~v_2\in\hat{\mathfrak a}_2.$$
Thus $\gamma(t)=e^{v_1(t)}\cdot e^{v_2(t)}$. By Lemma \ref{equi-aut}, we have
\begin{align}\label{eq-f-v1}
\nu(f((g_1,g_2)\gamma(t)))=\text{ord}_{t=0}f(e^{v_1(t)}(g_1,g_2)e^{v_2(t)} H_0).
\end{align}
By applying an action of $N_{\hat L}(\hat T)$, we may further assume that
$$\alpha_i(v_2)\leq0\text{ for }i=1,...,i_0.$$

On the other hand, the $\hat B^+$-eigenfunctions are of form
\begin{align}\label{eq-f-v2}
f_\lambda(g_1,g_2,w)=\langle v_{-\lambda},w^{-1}(g_2 g_1^{-1})v_\lambda\rangle,
\end{align}
where $\lambda$ is any weight in $\mathfrak M(\hat G/H_0)$ such that $\langle\lambda,\alpha_i\rangle\geq0,~i=1,...,i_0.$
By \eqref{eq-f-v1}-\eqref{eq-f-v2}, we see that for a generic choice of $(g_1,g_2)$,
$$\text{ord}_{t=0}f_\lambda(e^{v_1(t)}(g_1,g_2)e^{v_2(t)} H_0)=\lambda(v_1)+\lambda(v_2)=\lambda(v).$$
Also, by Proposition \ref{H0}, $\text{diag}(\exp((\Lambda_1-\Lambda_2)^\perp))\cdot e^{(\Lambda_1,\Lambda_2)t}$ acts trivially on $(\hat G\times \mathbb C^*)/\hat H_0$, so
\begin{align*}
&
\mathcal V(\hat G/ H_0)\\
=&\{-(y_1,y_2)\in\hat{\mathfrak a}|\alpha_{i}(y_2-y_1)\geq0,i=1,...,i_0\}/({\rm diag}((\Lambda_2-\Lambda_1)^\perp)\oplus\mathbb C(\Lambda_1,\Lambda_2))\\
\cong&\{y\in\mathfrak a|\alpha_{i}(y)\leq0,i=1,...,i_0\}.
\end{align*}
We conclude the Proposition since $\{\alpha_{i}|i=1,...,i_0\}=\Phi_{L,+,s}$ are precisely the simple roots in $\Phi_{L,+}$.
\end{proof}

Proposition \ref{valuation-cone} will be used to test the (modified) K-stability of $\mathcal X_0$ in Section 5.



\subsubsection{Moment polytope of $\mathcal X_0$}
By Proposition \ref{constrc-test-congifg}, when $\mathcal X$ is a special $\mathbb Z$-test configuration, $\mathcal X_0$ is a spherical embedding of $\hat G/H_0$. Recall \eqref{TC-divisor}. $(\mathcal X,\mathcal L)$ is a complete spherical embedding of $(\hat G\times\mathbb C^*)/({\rm diag}(G)\times\{e\})$ with moment polytope
\begin{align*}
\mathcal P:=\{m_0(y,y')\in\mathfrak a_+^*\oplus \mathbb R|0\leq y'\leq C_0-(\Lambda_1-\Lambda_2)(y)-m,~y\in P_+\}.
\end{align*}
Thus, if $\mathcal X$ is special, the central fibre $\mathcal X_0$ corresponds to the facet
$$\{m_0(y,y')\in\mathfrak a_+^*\oplus \mathbb R|y'= C_0-(\Lambda_1-\Lambda_2)(y)-m,~y\in P_+\}\subset\partial\mathcal P.$$
Hence we have
\begin{prop}\label{polytope-X0}
Suppose that $\mathcal X$ is a $\hat G$-equivariant special $\mathbb Z$-test configuration. Then there is a constant $C_0>0$ such that for each $k\in\mathbb N_+$ we can decompose $H^0(\mathcal X_0,K_{\mathcal X_0}^{-k})$ as direct sum of irreducible $\hat G\times \mathbb C^*$-representations
\begin{align*}
H^0(\mathcal X_0,K_{\mathcal X_0}^{-k})=\oplus_{\lambda\in \overline{kP_+}\cap\mathfrak M}V_\lambda\otimes V_{\lambda}^*\otimes E_{\frac 1m(kC_0-(\Lambda_1-\Lambda_2)(\lambda))},
\end{align*}
where $V_\lambda$ is the irreducible $G$-representation of highest weight $\lambda$, $E_q$ is the 1-dimensional representation of $\mathbb C^*$ of weight $q$. Consequently, the moment polytope of $(\mathcal X_0,K^{-1}_{\mathcal X_0})$ is $P_+$.
\end{prop}

\section{Filtrations and equivariant $\mathbb R$-test configurations}

In this section, we classify the $\hat G$-equivariant normal $\mathbb R$-test configurations of a polarized $G$-compactification. For simplicity, we write $\mathfrak M$ in short of $\mathfrak M(G)$. Recall that for a polarized $G$-compactification $(M,L)$ with moment polytope $P_+$, we can decompose $H^0(M,L^k)$ into direct sum of irreducible $\hat G$-representations \cite[Section 2.1]{AK},
\begin{align}\label{H0-M}
R_k=H^0(M,L^k)=\oplus_{\lambda\in\overline{k P_+}\cap\mathfrak M}{\rm End}(V_\lambda),
\end{align}
where $V_\lambda$ is the irreducible $G$-representation of highest weight $\lambda$ and ${\rm End}(V_\lambda)\cong V_\lambda\otimes V^*_\lambda.$ The Kodaira ring of $M$ is given by
\begin{align}\label{R}
R(M,L)=\oplus_{k\in\mathbb N}R_k.
\end{align}

Suppose that $\mathcal F$ is a $\hat G$-equivariant filtration on $R$. Then by Definition \ref{filtrantion-def} (1)-(2), we have
\begin{align}\label{F-s-k}
\mathcal F^sR_k=\oplus_{s^{(k)}_\lambda\geq s}{\rm End}(V_\lambda),
\end{align}
where we associated to each ${\rm End}(V_\lambda)$ in \eqref{H0-M} a number $s^{(k)}_\lambda$. Recall the Abelian group $\Gamma(\mathcal F)$ defined after \eqref{group-discon}. Under the normalization of Remark \ref{F-normalized}, we see that the corresponding Rees algebra \eqref{rees-def} reduces to
\begin{align}\label{rees}
{\rm R}(\mathcal F)=\oplus_{k\in\mathbb N}\oplus_{s\in\Gamma(\mathcal F)}\oplus_{\lambda\in\overline{k P_+}\cap\mathfrak M,s^{(k)}_\lambda\geq s}t^{-s}{\rm End}(V_\lambda).
\end{align}
$\mathcal F$ is an $\mathbb R$-test configuration if and only if \eqref{rees} is finitely generated.

For any two $\lambda_i\in\overline{k_i P_+}\cap\mathfrak M$ and $s_i$ such that ${\rm End}(V_{\lambda_i})\subset\mathcal F^{s_i}R_{k_i}$, $i=1,2$. By Definition \ref{filtrantion-def} (4),
$${\rm End}(V_{\lambda_1+\lambda_2})\subset{\rm End}(V_{\lambda_1})\cdot{\rm End}(V_{\lambda_2})\subset\mathcal F^{s_1}R_{k_1}\cdot\mathcal F^{s_2}R_{k_2}\subset\mathcal F^{s_1+s_2}R_{k_1+k_2}.$$
Thus, we have $s_{\lambda_1+\lambda_2}^{(k_1+k_2)}\geq s_1+s_2.$ In particular,
\begin{align}\label{alomost-concave}
s_{\lambda_1+\lambda_2}^{(k_1+k_2)}\geq s_{\lambda_1}^{(k_1)}+s_{\lambda_2}^{(k_2)}.
\end{align}

Also, Definition \ref{filtrantion-def} (3) is equivalent to that $\{s^{(k)}_\lambda/k\}_{\lambda\in\overline{kP_+}\cap\mathfrak M}$ are uniformly bounded with respect to all $k\in\mathbb N_+$. 

\subsection{Classification of $\hat G$-equivariant normal $\mathbb R$-test configurations}
Let $(M,L)$ be a polarized $G$-compactification. Recall \cite{AB2,AK}. The closure $Z$ of a maximal torus $T^\mathbb C$ in $M$, together with $L|_Z$ is a polarized toric variety . The polytope $P$ associated to $(Z,L|_Z)$ is a $W$-invariant, rational convex polytope in $\mathfrak M_\mathbb Q$. In fact, it holds $P_+=P\cap\overline{\mathfrak a^*_+}.$ Under the normalization of Remark \ref{F-normalized}, we have
\begin{theo}\label{G-classify}
Let $(M,L)$ be a polarized $G$-compactification with moment polytope $P_+$. Then for any $\hat G$-equivariant normal $\mathbb R$-test configuration $\mathcal F$ of $(M,L)$, there is a $W$-invariant, concave, piecewise-linear function $f$ on $\overline P$ whose domains of linearity consist of rational polytopes in $\mathfrak M_\mathbb Q$ such that $\min f=0$ and
\begin{align}\label{s-k-G-final}
s_{\lambda}^{(k)}=\max\{s\in\Gamma(\mathcal F)|s\leq kf(\frac\lambda k)\},~\forall \lambda\in\overline{kP_+}\cap\mathfrak M~\text{and}~k\in\mathbb N.
\end{align}
Moreover, $s_{\lambda}^{(k)}=kf(\frac\lambda k)$ if $\frac1k\lambda$ is a vertex of the domains of linearity of $f$.

Conversely, given any such $f$ and $r_0\in \mathbb N_+$ so that the domains of linearity of $r_0f(\frac\cdot {r_0})$ in $r_0P_+$ consist of integral polytopes in $\mathfrak M$. Denote by $\Gamma_{r_0}({\rm Vert}(f))$ the Abelian group generated by
$$\{r_0f(\frac1r_0\lambda)|\lambda~\text{is a vertex of a domain of linearity of $f$}\}.$$
Then for any finitely generated Abelian group $\Gamma$ containing $\Gamma_{r_0}({\rm Vert}(f))$, the collection of points of discontinuity
\begin{align}\label{s-k-G-general-inv}
s_{\lambda}^{(k)}=\sup\{s\in\Gamma|s\leq kf(\frac\lambda k)\},~\forall \lambda\in\overline{kP_+}\cap\mathfrak M~\text{and}~k\in\mathbb N.
\end{align}
together with \eqref{F-s-k} defines a $\hat G$-equivariant normal $\mathbb R$-test configuration $\mathcal F$ of $(M,L)$ satisfying  \eqref{s-k-G-final}.
\end{theo}


\begin{proof}
The proof will be divided into two parts according to the two directions.

\textbf{Part-1: From equivariant normal $\mathbb R$-test configurations to \eqref{s-k-G-final}.}

Suppose that $\mathcal F$ is a $\hat G$-equivariant normal $\mathbb R$-test configuration of $(M,L)$. By \eqref{total-space-ring},  up to replacing $L$ by some $L^{r_0}$ with $r_0\in\mathbb N_+$, we may assume that the Rees algebra ${\rm R}(\mathcal F)$ in \eqref{rees} is a normal ring.

We are going to construct a function $f$ satisfying \eqref{s-k-G-final}. As in \cite[Proposition 2.15]{Boucksom-Hisamoto-Jonsson}, for sufficiently large $r_0\in\mathbb N_+$ we may assume that the Rees algebra \eqref{rees} is generated by the piece $k=1$. We can choose $r_0$ sufficiently divisible so that even each vertex of $\overline {r_0P_+}$ is also integral. Also, without loss of generality, we can assume that $s_\lambda^{(1)}\geq0$ for all $\lambda\in\overline {P_+}\cap\mathfrak M$.

Let $\lambda\in\overline{kP_+}\cap\mathfrak M$ and $\mu_1,...,\mu_k\in\lambda\in\overline{P_+}\cap\mathfrak M$ so that
$${\rm End}(V_\lambda)\subset{\rm End}(V_{\mu_1})\cdot...\cdot{\rm End}(V_{\mu_k}).$$
By \eqref{alomost-concave} we have
$$s_\lambda^{(k)}\geq\sum_{j=1}^k s_{\mu_j}^{(1)}.$$
Since \eqref{rees} is generated by the piece $k=1$, we get for any $k\in\mathbb N_+$ and $\overline{kP_+}\cap\mathfrak M$,
\begin{align}\label{s-k-G}
s^{(k)}_\lambda=\max\{\sum_{i=1}^ks^{(1)}_{\mu_i}|&\mu_i\in\overline {P_+}\cap\mathfrak M,{\rm End}(V_\lambda)\subset{\rm End}(V_{\mu_1})\cdot...\cdot{\rm End}(V_{\mu_k})\}\geq0.
\end{align}

Let $\mathfrak M_z=\mathfrak M\cap\mathfrak {z(g)}$ and $\mathfrak M_{ad}={\rm Span}_\mathbb Z\Phi_{+,s}\oplus\mathfrak M_z.$
Also, let $\{\varpi_1,...,\varpi_r\}$ be the fundamental weights with respect to $\Phi_{+,s}=\{\alpha_1,...,\alpha_r\}$. It follows
\begin{align*}
\mathfrak M\subset\mathfrak M_{sc}={\rm Span}_\mathbb Z\{\varpi_1,...,\varpi_r\}\oplus\mathfrak M_z.
\end{align*}
Then there is an $n_G\in\mathbb N$ so that
\begin{align}\label{lattice-compare}
n_G\cdot\mathfrak M\subset n_G\cdot\mathfrak M_{sc}\subset\mathfrak M_{ad}\subset\mathfrak M\subset\mathfrak M_{sc}.
\end{align}

\emph{Step-1.1: Comparison of points of discontinuity.} We first show that for any $\lambda,\mu\in\overline{P_+}\cap\mathfrak M$ satisfying
\begin{align}\label{point-compare}
\lambda=\mu-\sum_{i=1}^rc_i\alpha_i,~0\leq c_i\in\mathbb Q~\text{for all}~i=1,...,r,
\end{align}
it holds
\begin{align}\label{s-compare}
s^{(1)}_\lambda\geq s^{(1)}_\mu.
\end{align}

Otherwise, if $s^{(1)}_\lambda<s^{(1)}_\mu$, by \eqref{F-s-k} we see that
\begin{align}\label{not-in-R}
t^{-s_\mu^{(1)}}{\rm End}(V_\lambda)\not\subset {\rm R}(\mathcal F).
\end{align}
Let $e_\lambda$ be the highest weight vector in ${\rm End}(V_\lambda)$. We see that for any $k\in\mathbb N_+$,
\begin{align}\label{power-in}
(t^{-s_\mu^{(1)}}e_\lambda)^{\cdot k}\in t^{-ks_\mu^{(1)}}{\rm End}(V_{k\lambda}).
\end{align}
Note that by \eqref{point-compare},$$\lambda\in{\rm Conv}\{w(\mu)|w\in W\}.$$
Hence by \cite[Lemma 1]{Timashev-Sbo}, there is a $k_0\in\mathbb N_+$ such that
$$V_{k_0n_G\lambda}\subset V_\mu^{\otimes k_0n_G}.$$
On the other hand, by \eqref{lattice-compare},
$$n_G(\mu-\lambda)\in{\rm Span}_{\mathbb Z_+}\Phi_{+,s}.$$
Combining with \cite[Proposition 4]{Timashev-Sbo}, we have
$${\rm End}(V_{k_0n_G\lambda})\subset{\rm End}(V_{\mu})^{\cdot k_0n_G}.$$
We get from \eqref{power-in} that
$$(t^{-s_\mu^{(1)}}e_\lambda)^{\cdot k_0n_G}\in (t^{-s_\mu^{(1)}}{\rm End}(V_{\mu}))^{\cdot k_0n_G}\subset {\rm R}(\mathcal F).$$
Thus $t^{-s_\mu^{(1)}}e_\lambda$ is integral over ${\rm R}(\mathcal F)$. Since ${\rm R}(\mathcal F)$ is normal, $$t^{-s_\mu^{(1)}}e_\lambda\in {\rm R}(\mathcal F),$$
which contradicts to \eqref{not-in-R} and we conclude \eqref{s-compare}.

\emph{Step-1.2: Construction of $f$.} In view of \eqref{point-compare}, for simplicity we will write ``$\mu\succeq\lambda$" whenever $\lambda$ and $\mu$ satisfy \eqref{point-compare}.
We claim that for each $k\in\mathbb N_+$, $\mu,\lambda_1,...,\lambda_l\in\overline{kP_+}\cap\mathfrak M$ and constants $0\leq c_1,...,c_l\leq1$ satisfying
\begin{align}
&\mu=\sum_{i=1}^lc_i\lambda_i,\label{mu-sum-lambda}
\end{align}
and
\begin{align}
&\sum_{1=1}^lc_i=1,\label{sum-ci}
\end{align}
it always holds
\begin{align}\label{s-k-concave}
s^{(k)}_\mu\geq\sup\{s\in\Gamma(\mathcal F)|s\leq\sum_{i=1}^lc_is^{(k)}_{\lambda_i}\}.
\end{align}

Otherwise, suppose that there are $\mu,\lambda_1,...,\lambda_l\in\overline{kP_+}\cap\mathfrak M$ and $0\leq c_1,...,c_l\leq1$ satisfying satisfying \eqref{mu-sum-lambda}-\eqref{sum-ci} but \eqref{s-k-concave} fails. Then
$$s^{(k)}_\mu<\sum_{i=1}^lc_i\hat s^{(k)}_{\lambda_i}=:\hat s^{(k)}_\mu\in\Gamma(\mathcal F),$$
where we can choose $\hat s^{(k)}_{\lambda_i}\in\Gamma(\mathcal F)_\mathbb Q$ with $\hat s^{(k)}_{\lambda_i}\leq s^{(k)}_{\lambda_i}.$ Let $e_\mu\in{\rm End}(V_\mu)$ be the highest weight vector, as in \eqref{not-in-R}, we have
\begin{align}\label{contra-1}
t^{-\hat s^{(k)}_\mu}e_\mu\not\in {\rm R}(\mathcal F).
\end{align}
On the other hand, we can choose $n_0,m_0\in\mathbb N_+$ so that $n_0c_j\in\mathbb N$ and $m_0\hat s_{\lambda_i}^{(k)}\in\Gamma(\mathcal F)$ for $j=1,...,l$. Then
\begin{align*}
(t^{-\hat s^{(k)}_\mu}e_\mu)^{\cdot n_0m_0}=&t^{-\sum_{j=1}^ln_0m_0c_j\hat s^{(k)}_{\lambda_j}}(e_\mu)^{\cdot n_0m_0}\in t^{-\sum_{j=1}^ln_0m_0c_j\hat s^{(k)}_{\lambda_j}} {\rm End}(V_{n_0m_0\mu}).
\end{align*}
However, by \eqref{mu-sum-lambda}-\eqref{sum-ci},
\begin{align*}
&t^{-\sum_{j=1}^ln_0m_0c_j s^{(k)}_{\lambda_j}}V_{n_0m_0\mu}\\
\subset& (t^{-m_0s^{(k)}_{\lambda_1}}{\rm End}(V_{m_0\lambda_1}))^{\cdot n_0c_1}\cdot...\cdot(t^{-m_0s^{(k)}_{\lambda_l}}{\rm End}(V_{m_0\lambda_l}))^{\cdot n_0c_l}
\subset{\rm R}(\mathcal F).
\end{align*}
Hence $(t^{-\hat s^{(k)}_\mu}e_\mu)^{\cdot n_0}$ is integral in ${\rm R}(\mathcal F)$. Since $s_{\lambda_i}^{(k)}m_0\leq s_{m_0\lambda_i}^{(m_0k)}$, we conclude that $t^{-\hat s^{(k)}_\mu}e_\mu\in{\rm R}(\mathcal F)$ since ${\rm R}(\mathcal F)$ is normal, which contradicts with \eqref{contra-1}. Hence \eqref{s-k-concave} is true.

Note that $s_\mu^{(1)}\in\Gamma(\mathcal F)$ for each $\mu\in\overline{P_+}\cap\mathfrak M$. Take $k=1$ in \eqref{s-k-G} and \eqref{s-k-concave}, we can define a piece-wise linear concave function $f:\overline{P_+}\to\mathbb R$ so that $\min f=0$ and
\begin{align}\label{s-1-G}
s_\mu^{(1)}=\max\{s\in\Gamma(\mathcal F)|s\leq f(\mu)\},~\forall \mu\in\overline{P_+}\cap\mathfrak M.
\end{align}
Indeed, consider the convex hull of
\begin{align}\label{graph-f}
P_+':=\{(\lambda,s^{(1)}_\lambda)|\lambda\in\overline{P_+}\cap\mathfrak M\}.
\end{align}
It is a convex polytope in $\overline{\mathfrak a_+^*}\oplus\mathbb R$ such that
$${\rm Conv}(P'_+)=\{(x,y)\in\overline{\mathfrak a_+^*}\oplus\mathbb R|x\in\overline{P_+},0\leq y\leq f(x)\}$$
for some concave function $f$. Obviously, $s^{(1)}_\mu\leq f(\mu)$ and the equality holds if $\mu$ is a vertex of a domain of linearity.
Combining with \eqref{s-k-concave}, we get \eqref{s-1-G}. Also, by the normalization condition $\min\Gamma(\mathcal F,1)=0$ and \eqref{s-1-G}, there must be a vertex $\mu$ of $P\cap P_+$ so that $s_\mu^{(1)}=0$. Hence $\min f=f(\mu)=0$.

It is easy to see that each domain of linearity of $f$ is a convex polytope with vertices in $\overline{P_+}\cap\mathfrak M$. In fact, any domains of linearity of $f$ is the projection of a facet of ${\rm Conv}(P'_+)$ on the roof $\{(x,f(x))|x\in\overline{P_+}\}$. But the vertices of such a facet must lie in $P'_+$. Hence it projects to a point in $\overline{P_+}\cap\mathfrak M$.

By \eqref{s-compare}, the gradient of $f$ (in the sense of subdifferential),
\begin{align}\label{nabla-f-1}
\nabla f\subset(-\overline{\mathfrak a_+}).
\end{align}
Hence by using the $W$-action we can extend $f$ to a $W$-invariant piecewise-linear function, which is globally concave on $\overline{P}$.

\emph{Step-1.3: Proof of \eqref{s-k-G-final}.} It remains to prove that \eqref{s-k-G-final} holds for any $k\in\mathbb N_+$. 
Fix any $\mu\in\overline{kP_+}\cap\mathfrak M$. By \eqref{s-k-G}, we can assume that there are $\lambda_1,...,\lambda_k\in\overline{P_+}\cap\mathfrak M$ so that
$$s_\mu^{(k)}=\sum_{i=1}^ks^{(1)}_{\lambda_i},~{\rm End}(V_\mu)\subset{\rm End}(V_{\lambda_1})\cdot...\cdot{\rm End}(V_{\lambda_k}).$$
In particular,
\begin{align}\label{sum-lambda-mu}
\sum_{i=1}^k\lambda_i\succeq\mu.
\end{align}

We have
\begin{align*}
\frac1k s_\mu^{(k)}=\frac1k \sum_{i=1}^ks^{(1)}_{\lambda_i}\leq\frac1k\sum_{i=1}^kf({\lambda_i})\leq f(\frac{\sum_{i=1}^k\lambda_i}k).
\end{align*}
By \eqref{sum-lambda-mu} and \eqref{nabla-f-1} we get
\begin{align}\label{s-leq-f}
\frac1k s_\mu^{(k)}\leq f(\frac{\sum_{i=1}^k\lambda_i}k)\leq f(\frac{\mu}k).
\end{align}

Recall that the vertices of each domain of linearity of $f$ is in $\mathfrak M$. Suppose that $\Omega$ is the domain of linearity which contains $\frac\mu k$ whose vertices are $\lambda_1,...,\lambda_l\in\mathfrak M$, then there are non-negative constants $c_1,...,c_l$ such that 
\begin{align*}
\mu=&\sum_{i=1}^lc_ik\lambda_i,~\sum_{1=1}^lc_i=1.
\end{align*}
Since $f$ is linear on $\Omega$, we get
\begin{align}\label{mu-lambda-comb}
k\sum_{j=1}^lc_js^{(1)}_{\lambda_j}=kf(\frac \mu k).
\end{align}
On the other hand, by \eqref{s-k-concave}, we have
\begin{align*}
s^{(k)}_\mu&\geq\sup\{s\in\Gamma(\mathcal F)|s\leq\sum_{j=1}^lc_js^{(k)}_{k\lambda_j}\}\\
&\geq\sup\{s\in\Gamma(\mathcal F)|s\leq k\sum_{j=1}^lc_js^{(1)}_{\lambda_j}\}=\sup\{s\in\Gamma(\mathcal F)|s\leq kf(\frac\mu k)\},
\end{align*}
where we used \eqref{s-k-G} and \eqref{mu-lambda-comb} in the last inequality.

Combining with \eqref{s-leq-f}, up to replacing $f(\cdot)$ by $\frac1{r_0}f(r_0\cdot)$ we get
\begin{align*}
s_{\lambda}^{(k)}=\sup\{s\in\Gamma(\mathcal F)|s\leq kf(\frac\lambda k)\},~\forall \lambda\in\overline{kP_+}\cap\mathfrak M~\text{and}~k\in\mathbb N.
\end{align*}
The relation \eqref{s-k-G-final} then follows from the above equality and the fact that $s^{(k)}_\lambda\in\Gamma(\mathcal F)$. Also, from \emph{Step-1.2} we see that after this scaling, $f$ is still a $W$-invariant, concave piecewise linear function on $P$. But its domains of linearity consists of convex polytopes with vertices in $\mathfrak M_\mathbb Q$.


\textbf{Part-2: The inverse direction.}


\emph{Step-2.1: Construction of the $\mathbb R$-test configuration.} Given an $f$ satisfying the assumption of Theorem \ref{G-classify}, we can fix an $r_0\in\mathbb N_+$ so that the domains of linearity of the function
$$\hat f(x)=r_0f(\frac 1{r_0}x):\overline{r_0P_+}\to\mathbb R$$
consist of polytopes with vertices in $\mathfrak M$. Replacing $L$ by $L^{r_0}$, we may assume that $r_0=1$. By concavity and \eqref{nabla-f-1}, it holds
$$(k_1+k_2)f(\frac\mu{k_1+k_2})\geq(k_1+k_2)f(\frac{\lambda_1+\lambda_2}{k_1+k_2})\geq k_1f(\frac{\lambda_1}{k_1})+k_2f(\frac{\lambda_2}{k_2})$$
for any two $\lambda_i\in\overline{k_iP_+}\cap\mathfrak M,\, i=1,2$ and $\mu\in\overline{(k_1+k_2)P_+}\cap\mathfrak M$ so that $\lambda_1+\lambda_2\succeq\mu$.
Combining with \cite[Proposition 4]{Timashev-Sbo}, it is direct to check that \eqref{F-s-k} and \eqref{s-k-G-general-inv} define a $\hat G$-equivariant filtration $\mathcal F$ of $(M,L)$. Then we prove that $\mathcal F$ is an $\mathbb R$-test configuration. We have two case:

\emph{Case-2.1.1.} $\Gamma$ is a discrete subgroup in $\mathbb R$. In this case \eqref{s-k-G-general-inv} is reduced to
\begin{align*}
s_{\lambda}^{(k)}=\max\{s\in\Gamma|s\leq kf(\frac\lambda k)\},~\forall \lambda\in\overline{kP_+}\cap\mathfrak M~\text{and}~k\in\mathbb N.
\end{align*}
Clearly $\Gamma(\mathcal F)\subset\Gamma$ is finitely generated and $\mathcal F$ is an $\mathbb R$-test configuration.

\emph{Case-2.1.2.} $\Gamma$ is a not discrete. Then $\Gamma$ is everywhere dense in $\mathbb R$. In this case \eqref{s-k-G-general-inv} is reduced to
\begin{align*}
s_{\lambda}^{(k)}=kf(\frac\lambda k),~\forall \lambda\in\overline{kP_+}\cap\mathfrak M~\text{and}~k\in\mathbb N.
\end{align*}
Consequently $\Gamma(\mathcal F)$ is generated by a finite set $\{f(\lambda)|\lambda\in\overline{P_+}\cap\mathfrak M\}$. Again $\mathcal F$ is an $\mathbb R$-test configuration.

\emph{Step-2.2: Normality of the total space.} It remains to prove that the corresponding Rees algebra
$${\rm R}(\mathcal F)=\oplus_{k\in\mathbb N}\oplus_{s\in\Gamma(\mathcal F)}\oplus_{\lambda\in\overline{k P_+}\cap\mathfrak M,kf(\frac1k\lambda)\geq s}t^{-s}{\rm End}(V_\lambda).$$
is normal. Since
$$({\rm R}(\mathcal F)\subset){\rm R}'=\oplus_{k\in\mathbb N}\oplus_{s\in\Gamma(\mathcal F)}\oplus_{\lambda\in\overline{k P_+}\cap\mathfrak M}t^{-s}{\rm End}(V_\lambda)$$
is a normal ring, it suffices to show that ${\rm R}(\mathcal F)$ is integrally closed in ${\rm R}'$. This is equivalent to show that any $\bar\sigma\in{\rm R}'\setminus{\rm R}(\mathcal F)$ is not integral in ${\rm R}(\mathcal F)$.

Assume there is some $\bar\sigma\in{\rm R}'\setminus{\rm R}(\mathcal F)$ integral over ${\rm R}(\mathcal F)$. It suffices to deal with $\bar\sigma$ of the following form:
\begin{align}\label{bar-sigma-decop}
\bar\sigma=\sum_{i=1}^{d}t^{-s_i}\sigma_{\tau_i},
\end{align}
where each $\tau_j\in\overline{k(j)P_+}\cap\mathfrak M$ for some $k(j)\in\mathbb N$,
\begin{align*}
(0\not=)\sigma_{\tau_j}\in{\rm End}(V_{\tau_j})~\text{and}~ s_j >k(j)f(\frac1{k(j)}\tau_j),~j=1,...,d.
\end{align*}
Since $\bar\sigma$ is integral over ${\rm R}(\mathcal F)$, there is some $q\in\mathbb N_+$ such that for any integer $l\in\mathbb N_+$,
\begin{align}\label{normal-condition}
\bar\sigma^l\in{\rm R}(\mathcal F)+{\rm R}(\mathcal F)\bar\sigma^1+...+{\rm R}(\mathcal F)\bar\sigma^{q}.
\end{align}

Consider the convex hull of $\{w(\tau_i)|w\in W,~i=1,...,d\}$. Choose a $\tau\in\{\tau_1,\dots, \tau_d\}$ which is a vertex of this convex hull. Then for any $p\in\mathbb N_+$, $p\tau$ can not be dominated by $\sum_{s=1}^p\tau_{i_s}$ whenever there is a $\tau_{i_s}\in\{\tau_1,\dots, \tau_d\}\setminus\{\tau\}$. Take the component $t^{-s}\sigma_{\tau}$ of $\bar{\sigma}$ in its decomposition \eqref{bar-sigma-decop}. 
Suppose that the degree of $t^{-s}\sigma_{\tau}$ is $m$ (that is, $\sigma_\tau\in R_m$). Then
\begin{align}\label{s-lower}
m(f(\frac1m\tau))<s.
\end{align}
For any $l\in\mathbb N_+$, we can decompose $\bar\sigma^l$ as \eqref{bar-sigma-decop}. By Lemma \ref{alg-lem} below, there is a nonzero component of $\bar\sigma^l$ in $t^{-ls}{\rm End}(V_{l\tau})$ with degree $lm$, we denote it by $t^{-ls}\sigma_{l\tau}$.

Decompose $\bar{\sigma}^i, i=0,\dots,q$ as \eqref{bar-sigma-decop}. Then all their components have the form $t^{-w}\sigma_{\gamma}$ for some degree $k$. All such triples $(w,\gamma, k)$ form a finite set $\mathcal S$. Thus
$$\bar\sigma^l\in\sum_{(w,\gamma,k)\in\mathcal S}{\rm R}(\mathcal F)t^{-w}\sigma_{\gamma}.$$

Consider the component $t^{-ls}\sigma_{l\tau}$ of $\bar\sigma^l$. Since $\mathcal S$ is finite, up to passing to a subsequence, there is some $(w,\gamma, k)\in\mathcal S$ such that
$$t^{-ls}\sigma_{l\tau}\in t^{-w}{\rm End}(V_{\gamma}){\rm R}(\mathcal F)_{lm-k},~l\in\mathbb N_+.$$
Thus, we have
$$t^{-ls}\sigma_{l\tau}\in t^{-w}{\rm End}(V_{\gamma})t^{-r}{\rm End}(V_{\mu})$$ for some $t^{-r}{\rm End}(V_{\mu})\subseteq {\rm R}(\mathcal F_f)_{lm-k}$, where $ls=w+r,$
\begin{align}\label{gamma+mu}
\gamma+\mu\succeq l\tau,
\end{align}
and
\begin{align*}
r\leq (lm-k)f(\frac{\mu}{lm-k}).
\end{align*}
Here in \eqref{gamma+mu} we used \cite[Proposition 4]{Timashev-Sbo}. Hence
\begin{align*}
s=\frac wl+\frac rl\leq& \frac wl+(m-\frac kl)f(\frac1{lm-k}{\mu})\\
\leq&\frac wl+(m-\frac kl)f(\frac{\tau-\frac1l{\gamma}}{m-\frac kl}),~\text{for sufficiently large}~l\in\mathbb N_+,
\end{align*}
where in the last line we used \eqref{gamma+mu} and \eqref{nabla-f-1}. Sending $l\to+\infty$  we see that $$s\leq mf(\frac1m\tau),$$ which contradicts to \eqref{s-lower}. Hence \eqref{normal-condition} is not true and we conclude that ${\rm R}(\mathcal F)$ is normal.

\end{proof}

To complete the arguments in \emph{Step-2.2}, we need the following

\begin{lem}\label{alg-lem}
Suppose that $\sigma\in{\rm End}(V_\mu)\subset R_k$
for some $k\in\mathbb N$. Then for any $q\in\mathbb N_+$, $\sigma^q$ has non-zero component in ${\rm End}(V_{q\mu})$.
\end{lem}

\begin{proof}
Since ${\rm End}(V_\mu)$ is an irreducible $\hat G$-representation, there is a $\hat g_0\in\hat G$ so that $\hat g_0(\sigma)$ has non-zero component on the subspace generated by the highest weight vector. Thus for any $q\in\mathbb N_+$, $\hat g_0(\sigma^q)$ has non-zero component in ${\rm End}(V_{q\mu})$. We conclude the Lemma since
$\sigma^q=\hat g_0^{-1}(\hat g_0(\sigma^q))$ and the fact that ${\rm End}(V_{q\mu})$ is $\hat G$-invariant.
\end{proof}

\begin{rem}\label{ZTC}
When $f$ is rational, we can take $k$ to be the smallest positive integer so that the set
$$\{k(\lambda,s)|0\leq s\leq f(\lambda),~\lambda\in\overline{P_+}\}\subset\mathfrak M_\mathbb R\oplus\mathbb R$$
is an integral polytope and $\Gamma=\mathbb Z$. Theorem \ref{G-classify} then reduces to the classification theorem of $\hat G$-equivariant normal $\mathbb Z$-test configurations \cite[Section 2.4]{AK} (based on \cite[Section 4.2]{AB2}).
\end{rem}


We can further classify $\hat G$-equivariant normal $\mathbb R$-test configurations with reduced central fibre by using Theorem \ref{G-classify}.

\begin{theo}\label{G-classify-reduced}
Let $(M,L)$ be a polarized $G$-compactification with moment polytope $P_+$. Then for any $\hat G$-equivariant normal $\mathbb R$-test configuration $\mathcal F$ of $(M,L)$ with reduced central fibre, there is a $W$-invariant, concave, piecewise-linear function $f\geq\min f=0$ on $\overline P$ whose domains of linearity consist of rational polytopes in $\mathfrak M_\mathbb Q$, such that
\begin{align}\label{s-k-G-final-reduced}
s_{\lambda}^{(k)}=kf(\frac\lambda k),~\forall \lambda\in\overline{kP_+}\cap\mathfrak M~\text{and}~k\in\mathbb N,
\end{align}
and vice versa.
\end{theo}

Theorem \ref{RTC-classify} follows directly from Theorem \ref{G-classify-reduced}. In the following we will call $f$ in \eqref{s-k-G-final-reduced} the \emph{function associated to $\mathcal F$} and denote $\mathcal F=\mathcal F_f$.

\begin{proof}[Proof of Theorem \ref{G-classify-reduced}]
We divide the proof in two parts.

\textbf{Part-1: Necessity of \eqref{s-k-G-final-reduced}.} Suppose that $\mathcal F$ is given so that ${\rm Gr}(\mathcal F)$ defined by \eqref{GrF-def} contains no nilpotent element. Let $f$ be the function defined in Theorem \ref{G-classify}. We will show \eqref{s-k-G-final-reduced} holds. Otherwise, by \eqref{s-k-G-final} there is a $\lambda_0\in\overline{l_0P_+}\cap\mathfrak M$ for some $l_0\in\mathbb N_+$ so that
$$(\Gamma(\mathcal F)\ni)s_{\lambda_0}^{(l_0)}<l_0f(\frac{\lambda_0}{l_0}).$$
Let $\sigma_0\in{\rm End}(V_{\lambda_0})$ be a highest weight vector. Then $\sigma_0^{\otimes k}\in{\rm End}(V_{k\lambda_0})$ has (real) weight $t^{-ks_{\lambda_0}^{(l_0)}}$ for any $k\in\mathbb N_+$.

On the other hand, choose a set of positive generators $\{s_1,...,s_{r_{\mathcal F}}\}$ of ${\mathcal F}$ and $s_M:=\max\{s_1,...,s_{r_{\mathcal F}}\}$. Then there is a $k_0\in\mathbb N_+$ so that
$$k_0(l_0f(\frac{\lambda_0}{l_0})-s_{\lambda_0}^{(l_0)})\geq s_M.$$
Hence there must be an $s'\in\Gamma(\mathcal F)$ so that $k_0s_{\lambda_0}^{(l_0)}<s<k_0l_0f(\frac{\lambda_0}{l_0})$ and consequently
$$s_{k_0\lambda_0}^{(k_0l_0)}>k_0s_{\lambda_0}^{(l_0)}.$$
Hence ${\rm End(V_{k_0\lambda_0})}\subset \mathcal F^{>k_0s_{\lambda_0}^{(l_0)}}R_{k_0l_0}$ and in \eqref{GrF-def} the piece
$$t^{-k_0s^{(l_0)}_{\lambda_0}}\mathcal F^{k_0s_{\lambda_0}^{(l_0)}}R_{k_0l_0}/\mathcal F^{>k_0s_{\lambda_0}^{(l_0)}}R_{k_0l_0}$$
contains no ${\rm End}(V_{k_0\lambda_0})$-factor. Hence $\sigma_0^{\otimes k}$ descends to $0$ in ${\rm Gr}(\mathcal F)$. In other words, $\sigma_0^{\cdot k}=0$ in ${\rm Gr}(\mathcal F)$ and $\sigma_0$ is nilpotent, a contrdiction.

\textbf{Part-2: Sufficiency of \eqref{s-k-G-final-reduced}.} Suppose that \eqref{s-k-G-final-reduced} holds. Define
$$\mathcal F^sR_k:=\oplus_{\lambda\in\overline{kP_+}\mathfrak M,kf(\lambda/k)\geq s}{\rm End}(V_\lambda),~\forall k\in\mathbb N_+.$$
We will show \eqref{GrF-def} contains no nilpotent element. Otherwise, there are a $\sigma\in{\rm End}(V_{\lambda_0})$ for some $\lambda_0\in \overline{l_0P_+}\cap\mathfrak M$ and some $k_0\in\mathbb N_+$ so that $\sigma_0^{\cdot k}=0$ in ${\rm Gr}(\mathcal F)$.

By Lemma \ref{alg-lem}, we can assume that $\sigma$ has non-zero component on the direction of highest weight vector. Hence $\sigma_0^{\cdot k}$ has non-zero component in  $\sigma\in{\rm End}(V_{k_0\lambda_0})$. It must hold
$$t^{-k_0s^{(l_0)}_{\lambda_0}}\mathcal F^{k_0s_{\lambda_0}^{(l_0)}}R_{k_0l_0}/\mathcal F^{>k_0s_{\lambda_0}^{(l_0)}}R_{k_0l_0}$$
contains no ${\rm End}(V_{k_0\lambda_0})$-factor. This implies
$$s_{k_0\lambda_0}^{(k_0l_0)}>k_0s_{\lambda_0}^{(l_0)},$$
a contradiction to \eqref{s-k-G-final-reduced}.
\end{proof}

\begin{rem}\label{RTC-reduced}
Given any $f$ satisfying the assumption of Theorem \ref{G-classify-reduced}, we can construct $\mathcal F_f$ by choosing the following data in Theorem \ref{G-classify}:
\begin{itemize}
\item $k$ is the smallest positive integer so that the domains of linearity of $f$ in $P_+$ consists of integral polytopes in $\mathfrak M$.
\item $\Gamma$ is the group generated by $\{s_{\lambda}^{(k)}|s_{\lambda}^{(k)}~\text{is given by}~\eqref{s-k-G-final-reduced}\}$.
\end{itemize}
It is direct to check that $\Gamma(\mathcal F_f)=\Gamma$.
\end{rem}

\begin{coro}\label{R-TC-Lambda}
Let $(M,L)$ be a polarized $G$-compactification. Then for any $\Lambda\in\overline{\mathfrak a_+}$ there is a $\hat G$-equivariant special $\mathbb R$-test configuration $\mathcal F_{\Lambda}$ of $(M,L)$, so that the central fibre $\mathcal X_0$ is a $\hat G$-spherical variety and admits an action of the torus $\overline{\exp(t\Lambda)}\subset\hat G$. Conversely, for any $\hat G$-equivariant special $\mathbb R$-test configuration $\mathcal F$ of $(M,L)$, there is a $\Lambda\in\overline{\mathfrak a_+}$ such that $\mathcal F=\mathcal F_{\Lambda}$.
\end{coro}
\begin{proof}
It suffices to show that a $\hat G$-equivariant normal $\mathbb R$-test configuration $\mathcal F$ is special if and only if the associated function $f$ given by Theorem \ref{G-classify-reduced} is affine on $\overline {P_+}$. Given any $\Lambda\in\overline{\mathfrak a_+}$, define
\begin{align}\label{affine-f}
f_\Lambda=\max_{\mu\in\overline{P_+}}\Lambda(\mu)-\Lambda(\lambda).
\end{align}
Let $\mathcal F_\Lambda$ be the $\hat G$-equivariant normal $\mathbb R$-test configuration associated to $f_\Lambda$ defined by Theorem \ref{G-classify-reduced}. Then its centre $\mathcal X_0$ is reduced. 

As in \cite[Section 2.2]{Han-Li} (see \cite[end of p.9 - beginning of p.10]{Han-Li}), we can perturb $\Lambda$ to some $\Lambda'\in\mathfrak N_\mathbb Q$ so that $\mathcal F_\Lambda$ and $\mathcal F_{\Lambda'}$ have the same central fibre. In fact, $\Lambda'$ can be any rational vector in the Lie algebra of $\overline{\exp(t\Lambda)}$ sufficiently close to $\Lambda$. Hence by Proposition \ref{constrc-test-congifg}, $\mathcal X_0$ is irreducible. Note that when \eqref{s-k-G-final-reduced} holds, $\mathcal X_0$ is also reduced. By \cite[Theorem 15.20]{Timashev-book} $\mathcal X_0$ is normal, in particular it is a $\hat G$-spherical variety. We see that $\mathcal F_\Lambda$ is special.

When \eqref{affine-f} holds, by \eqref{GrF-def} and \eqref{s-k-G-final-reduced} we see that the vector field induced by $\mathcal F_\Lambda$ is $\Lambda$. Hence we get the first part of the Corollary.

Conversely, suppose that $\mathcal F$ is a $\hat G$-equivariant special $\mathbb R$-test configuration. Then $\mathcal X_0$ is reduced and irreducible. Since the Rees algebra ${\rm R}(\mathcal F)$ is finitely generated, we can assume that it can be generated by
$$\oplus_{0\leq k\leq r_0,k\in\mathbb N_+}\oplus_{s\in\Gamma(\mathcal F,k)}\oplus_{\lambda\in\overline{k P_+}\cap\mathfrak M,s^{(k)}_\lambda\geq s}t^{-s}{\rm End}(V_\lambda).$$
Hence we can perturb each $s^{(k)}_\lambda$ for $k\in[0,...,r_0]$ and $\lambda\in\overline{kP_+}\cap\mathfrak M$ to a rational number $s'^{(k)}_\lambda$ sufficiently close to it so that it defines a $\mathbb Z$-test configuration $\mathcal F'$ with the same central fibre. By Proposition \ref{parameter-of-test-config}, the function $f'$ associated to $\mathcal F'$ is affine on $\overline{P_+}$.\footnote{Precisely, to apply Proposition \ref{parameter-of-test-config} one needs to rescale $\Gamma(\mathcal F')$ so that it coincides with the standard lattice $\mathbb Z$. Hence $(\Lambda,0,m)$ in Proposition \ref{parameter-of-test-config} should be taken as $(k'\nabla f',0,1)$ for some sufficiently large $k'$.} As for small perturbations, the number of domains of linearity of $f'$ can not be smaller than that of $f$, we get the Corollary.
\end{proof}

There is also an algebraic proof of Corollary \ref{R-TC-Lambda} in the Appendix. See Proposition \ref{prop-app} below.

\subsection{Approximation of an equivariant $\mathbb R$-test configuration}
In the following we approximate a $\hat G$-equivariant normal $\mathbb R$-test configuration by a sequence of $\mathbb Z$-test configurations as \cite[Definition-Proposition 2.15]{Han-Li}. More precisely, given such an $\mathbb R$-test configuration $\mathcal F$, we can construct a sequence of $\hat G$-equivariant normal $\mathbb Z$-test configurations $\{\mathcal F_p\}_{p\in\mathbb N_+}$ so that
the filtration of $\mathcal F_p$ on $\oplus_{k\in\mathbb N}R_{pk}$ is induced by $\mathcal F_\mathbb Z$ on the $R_p$-piece (cf. \cite[Section 2.2]{Han-Li}),
$$(\mathcal F_\mathbb Z)^sR_{p}=\mathcal F^{\lceil s\rceil}R_{p}.$$
Indeed, let $f$ be the functions associated to $\mathcal F$ given by Theorem \ref{G-classify}. Define
$$f_p(\mu)=\min\{\varphi(\mu)|\varphi(\mu)~\text{is concave and}~\varphi(\frac1p\lambda)\geq\frac{[s^{(p)}_\lambda]}{p},~\lambda\in\overline{pP_+}\cap\mathfrak M\}.$$
Then $\mathcal F_p$ is the $\mathbb Z$-test configuration defined by $f_p$ (cf. Remark \ref{ZTC}).

Choose a set of nonnegative generators $\{e_j\}_{j=1}^{r_\mathcal F}$ of $\Gamma(\mathcal F)$ and denote $\delta_{\mathcal F}=\max_j\{e_j\}$. Then $$0\leq f(\lambda)-\frac{1+\delta_\mathcal F}{p}\leq \frac{[s^{(p)}_\lambda]}{p}\leq f(\lambda),$$
we have
\begin{align}\label{fp-con-1}
0\leq f(\lambda)-f_p(\lambda)\leq \frac{1+\delta_\mathcal F}{p},~\lambda\in\overline {P_+}.
\end{align}
Hence we get the uniform convergency,
\begin{align}\label{fp-con}
f_p\rightrightarrows f,~\text{on}~\overline {P_+}~\text{as}~p\to+\infty.
\end{align}

\section{H-invariant and semistable limit}
In this section, we estimate the H-invariant of a general $\hat G$-equivariant normal $\mathbb R$-test configuration of a $\mathbb Q$-Fano $G$-compactification. In particular we compute its precise value for $\hat G$-equivariant special $\mathbb R$-test configurations. Then we find its minimizer and prove Theorem \ref{main-thm-1}.

\subsection{Reduction of the H-invariant}

In this section we express the H-invariant of an equivariant normal $\mathbb R$-test configuration in terms of its associated function.

\begin{theo}\label{H-f-reduction}
Let $(M,L)$ be a $\mathbb Q$-Fano $G$-compactification with moment polytope $P_+$. Let $\mathcal F$ a $\hat G$-equivariant normal $\mathbb R$-test configuration of $(M,K_M^{-1})$ and $f$ the function defined in Theorem \ref{G-classify}. Then up to adding a uniform constant,
\begin{align}\label{H-reduction-eq}
H(\mathcal F)\geq\ln\left(\frac1V\int_{P_+}e^{-f(y)+f(2\rho)}\pi(y)dy\right),
\end{align}
and the equality holds if $\mathcal F$ is special.
\end{theo}
Note that a special $\mathbb R$-test configuration can be defined by a valuation on $G$ and has reduced central fibre. The proof of the Theorem is a combination of \eqref{H-inv-NA} and the following Lemmas \ref{S-F-lem-app}-\ref{L-F-lem-app}.

\begin{lem}\label{S-F-lem-app}
Under the assumption of Theorem \ref{H-f-reduction}, up to adding a uniform constant,
\begin{align*}
S^{\rm NA}(\mathcal F)\leq-\ln\left(\frac1V\int_{P_+}e^{-f(y)}\pi(y)dy\right),
\end{align*}
and the equality holds if $\mathcal F=\mathcal F_f$, the $\hat G$-equivariant normal $\mathbb R$-test configuration with reduced central fibre defined in Theorem \ref{G-classify-reduced}.
\end{lem}

\begin{proof}
Recall \eqref{S-inv-NA}. We need to find the Okounkov bodies. By \eqref{H0-M}, \eqref{F-s-k} and \eqref{s-k-G-final}, we have
$$\mathcal F^{kt}R_k=\oplus_{\lambda\in\overline{kP_+}\cap\mathfrak M,s^{(k)}_\lambda\geq kt}{\rm End}(V_\lambda),~\forall k\in\mathbb N~\text{and}~t\in\mathbb R.$$
On the other hand, by \eqref{s-k-G-final}, $s^{(k)}_\lambda\leq kf(\frac1k\lambda)$. Thus the Okounkov bodies
$$\Delta(\mathcal F^{kt}R_{k})\subset{\rm Conv}\left(\cup_{\lambda\in\overline{kP_+}\cap \mathfrak M;f(\lambda/k)\geq t}(\lambda,\Delta(\lambda))\right),$$
and the equality holds if \eqref{s-k-G-final-reduced} is true. Combining with \eqref{okounkov-body-layer} and \eqref{okounkov-body}, the Okounkov body of $\mathcal F_{\Lambda}^{(t)}:=\{\mathcal F_\Lambda^{tk}R_k\}_{k\in\mathbb N_+}$ is
\begin{align}\label{okounkov-body-serise}
\Delta(\mathcal F^{(t)})=\overline{{\rm Conv}\left(\cup_{k=1}^{+\infty}\frac1k\Delta(\mathcal F^{kt}R_{k})\right)}\subset\Delta\cap\{f(\lambda)\geq t\}=:\Delta_{f\geq t},
\end{align}
and equality holds if \eqref{s-k-G-final-reduced} is true.

Recall \eqref{okounkov-body}. Each $z\in\Delta$ can be decomposed as $z=(\lambda,z')$, where $\lambda\in\overline{P_+}$ and $z'\in\mathbb R^{\dim(\hat N_u)}$.
Set $\Delta_\lambda=\{z'|(\lambda,z')\in\Delta\}.$ By \eqref{okounkov-body-func} and \eqref{okounkov-body-serise},
\begin{align}\label{G-F-Lambda}
G_{\mathcal F}(z)\leq\sup\{t|z\in\Delta_{f\geq t}\}=f(\lambda),~\text{for}~z=(\lambda,z')\in\Delta_\lambda\subset\Delta,
\end{align}
with the equality holds for all $z\in\Delta$ if \eqref{s-k-G-final-reduced} holds.

We want to decompose the measure $dz$ on $\Delta$. By \cite[Theorem 2.5]{Han-Li}, the Dirac type measure
\begin{align}\label{dirac-measure}
\nu_k:=\frac{n!}{k^n}\sum_{z\in\Delta~\text{is an integral point}~}\delta_{\frac zk}
\end{align}
converges weakly to $dz$ on $\Delta$,
\begin{align}\label{dirac-measure-converge}
dz=\lim_{k\to+\infty}\nu_k.
\end{align}

We may rewrite \eqref{dirac-measure} as
$$\nu_k=\frac{n!}{k^n}\sum_{\lambda\in\overline{P_+}\cap\frac1k\mathfrak M}\left(\sum_{z'\in\Delta(k\lambda)~\text{is an integral point}~}\delta_{(\lambda,\frac {z'}k)}\right).$$
Recall the Weyl character formula \cite[Section 3.4.4]{Zhelobenko-Shtern},
$$\dim(V_\lambda\otimes V_\lambda^*)=\frac{\prod_{\alpha\in\Phi_+}\langle\alpha,\rho+k\lambda\rangle^2}{\prod_{\alpha\in\Phi_+}\langle\alpha,\rho\rangle^2},~\forall\lambda\in\overline{\mathfrak a_+}\cap\mathfrak M.$$
By \eqref{okounkov-body}, for any continuous function $\varphi$ on $\Delta$, which only depends on $\lambda\in\overline{P_+}$, we have
\begin{align*}
\int_{\Delta}\varphi\nu_k=&\frac{n!}{k^n}\sum_{\lambda\in\overline{P_+}\cap\frac1k\mathfrak M}\left(\sum_{z'\in\Delta(k\lambda)~\text{is an integral point}~}\varphi(\lambda)\right)\notag\\
=&\frac{n!}{k^n}\sum_{\lambda\in\overline{P_+}\cap\frac1k\mathfrak M}\varphi(\lambda)\frac{\prod_{\alpha\in\Phi_+}\langle\alpha,\rho+k\lambda\rangle^2}{\prod_{\alpha\in\Phi_+}\langle\alpha,\rho\rangle^2}.
\end{align*}
Sending $k\to+\infty$, by \eqref{dirac-measure-converge} and \cite[Section 1.4]{Pukhlikov-Khovanskii},
\begin{align}\label{dz}
\int_{\Delta}\varphi dz=\int_{P_+}\varphi(\lambda)\frac{\pi(\lambda)}{\prod_{\alpha\in\Phi_+}\langle\alpha,\rho\rangle^2}d\lambda.
\end{align}

By \eqref{S-inv-NA}, \eqref{G-F-Lambda} and \eqref{dz}, we have
\begin{align}\label{S-int-F-Lambda}
S^{\rm NA}(\mathcal F)=&-\ln\left(\frac1V\int_{\Delta}e^{-G_{\mathcal F}(z)}dz\right)\notag\\
\leq&-\ln\left(\frac1V\int_{P_+}e^{-f(\lambda)}\pi(\lambda)d\lambda\right)+\ln{\prod_{\alpha\in\Phi_+}\langle\alpha,\rho\rangle^2},
\end{align}
and the equality holds provided \eqref{s-k-G-final-reduced}. Since $\ln{\prod_{\alpha\in\Phi_+}\langle\alpha,\rho\rangle^2}$ is a constant depending only on $G$, we get the Lemma.
\end{proof}

\begin{lem}\label{L-F-lem-app}
Under the assumption of Theorem \ref{H-f-reduction}, we have
\begin{align*}
L^{\rm NA}(\mathcal F)\geq f(2\rho),
\end{align*}
and the equality holds if $\mathcal F$ is defined by a valuation on $G$.
\end{lem}

\begin{proof}
To compute $L^{\rm NA}(\mathcal F)$, we first deal with the case when $f$ is rational. In this case, $\mathcal F$ is a $\mathbb Z$-test configuration of some exponent $m_0\in\mathbb N_+$. The $L^{\rm NA}$-functional in \eqref{L-inv-NA} is computed by the log-canonical threshold (cf. \cite[Example 2.31]{Han-Li}),
\begin{align}\label{L-inv-NA-lct}
L^{\rm NA}(\mathcal F_\Lambda)={\rm lct}_{(\mathcal X,-(\frac1{m_0}\mathcal L+K_\mathcal X-\pi_\mathcal X^*K_{\mathbb {CP}^1}))}(\mathcal X_0)-1.
\end{align}

Since $f$ is rational, we may assume that $f$ has domains of linearity $\Omega_1,...,\Omega_{N_f}$ so that
$$f(y)=C_a-\Lambda_a(y), ~\forall y\in\Omega_a,a=1,...,N_f,$$
where each $\Lambda_a\in\overline{\mathfrak a_+}\cap\mathfrak N_\mathbb Q$ and $C_a\in\mathbb Q$. Consequently,
$$f(y)=\min_a\{C_a-\Lambda_a(y)\}, ~\forall y\in\overline{P_+}.$$
Without loss of generality we may assume that $f\geq0$. Set
$$\mathcal P_+=\{(y,t)|y\in P_+,0\leq t\leq f(y)\},$$
Then $m_0\mathcal P_+$ is an integral polytopewhich is the moment polytope of $(\mathcal X,\mathcal L)$.

Let $\{F_A\}_{A=1,...,d_+}$ be the outer facets of $P_+$. Then each facet
$$\hat F_A=\{(y,t)\in m_0\mathcal P_+|y\in m_0F_A\},~A\in\{1,...,d_+\},$$
of $m_0\mathcal P_+$ corresponds to a $\hat G\times \mathbb C^*$-invariant divisors $\hat Y_A$ of $\mathcal X$. There are also prime boundary divisors $\hat Y_\infty$ that corresponds to $m_0P_+\cap\{0\}$ and $\hat Y_{0,a},a=1,...,N_f$ that corresponds to the piece
$$\{(y,t)\in m_0\mathcal P_+|t=m_0C_a-\Lambda_a(y),y\in\overline{\Omega_a}\}.$$
Recall that $(\mathcal X,\mathcal L)$ is a polarized $\hat G\times \mathbb C^*$-compactification. The colours are exactly given by the closures
$$\hat D_\alpha=\overline{D_\alpha\times \mathbb C^*},~D_\alpha\text{ is a colour of }M.$$
Let $Y_A$ be the $\hat G$-invariant divisor of $M$ that corresponds to $F_A$. Since the divisor $$-K_M=\sum_AY_A+2\sum_\alpha D_\alpha,$$
we get
\begin{align}\label{-K-mathcal-X}
-K_{\mathcal X}=\sum_A\hat Y_A+\sum_a\hat Y_{0,a}+\hat Y_\infty+2\sum_\alpha \hat D_\alpha.
\end{align}
For $a=1,...,N_f$, denote by $m_a$ the smallest positive integer so that $m_a\Lambda_a\in\mathfrak N$. As in \cite[Proof of Theorem 14]{Yao}, the pull-back of $-K_{\mathbb {CP}^1}$ by the projection $\pi_\mathcal X:\mathcal X\to\mathbb {CP}^1$ is
\begin{align}\label{pi-K-CP1}
-\pi^*_\mathcal XK_{\mathbb{CP}^1}=\hat Y_\infty+\sum_am_a\hat Y_{0,a},
\end{align}
and
\begin{align}\label{cal-X-0}
\mathcal X_0=\sum_am_a\hat Y_{0,a}.
\end{align}


On the other hand, note that the Cartier line bundle $K_M^{-m_0}$ has a $\hat B$-semi-invariant section of weight $2m_0\rho$ \cite[Section 3.2.4]{Del3}. As in Section 3.1, $\mathcal L$ has a $\hat B\times \mathbb C^*$-semi-invariant section of weight $2m_0\rho$ whose divisor is
\begin{align*}
\mathcal L=m_0(\sum_A\hat Y_A+\sum_am_a(C_a-2\Lambda_a(\rho))\hat Y_{0,a}+2\sum_\alpha \hat D_\alpha).
\end{align*}
Combing this with \eqref{-K-mathcal-X}-\eqref{cal-X-0}, we get
\begin{align*}
D_c=&-(\frac1{m_0}\mathcal L+K_\mathcal X-\pi_\mathcal X^*K_{\mathbb {CP}^1})+c\mathcal X_0\\
=&\sum_a(1+cm_a-(C_a-2\Lambda_a(\rho)+1)m_a)\hat Y_a.
\end{align*}
Recall that $(\mathcal X,-K_\mathcal X)$ is always a log canonical pair \cite[Section 5]{AB2}. We get
\begin{align*}
{\rm lct}_{(\mathcal X,D_0)}(\mathcal X_0):=&\sup\{c|(\mathcal X,D_c)~\text{is sublc}\}\\
=&\sup\{c|(1+cm_a-(C_a-2\Lambda_a(\rho)+1)m_a)\leq1,~a=1,...,N_f\}\\
=&\min_a(C_a-2\Lambda_a(\rho)+1)=1+f(2\rho)=f(2\rho)+1.
\end{align*}
Thus $$L^{\rm NA}(\mathcal F)=f(2\rho)$$ for any $\hat G$-equivariant normal $\mathbb Z$-test configuration $\mathcal F$.

For a general $G\times G$-equivariant normal $\mathbb R$-test configuration $\mathcal F$ with associated function $f$, we can choose a sequence of approximating $G\times G$-equivariant normal $\mathbb Z$-test configurations $\mathcal F_p$ with associated function $f_p$ constructed in Section 4.2. By the above computation,
\begin{align}\label{L-Z-test}
L^{\rm NA}(\mathcal F_p)=f_p(2\rho),~p\in\mathbb N_+.
\end{align}
On the other hand, by \cite[Remark 2.29, 3.32]{Han-Li},
\begin{align*}
\lim_{p\to+\infty}L^{\rm NA}(\mathcal F_p)\leq L^{\rm NA}(\mathcal F),
\end{align*}
with the equality holds if $\mathcal F$ is defined by a valuation on $G$. Combining with \eqref{fp-con} and \eqref{L-Z-test} we get the Lemma.

\end{proof}

\subsection{Minimizer of the H-invariant}

In this section we study the minimizer of the H-invariant and finish the proof of Theorem \ref{main-thm-1}.


\begin{proof}[Proof of Theorem \ref{main-thm-1}]
Let $\mathcal F$ be any $\hat G$-equivariant normal $\mathbb R$-test configuration  and $f$ the function defined in Theorem \ref{G-classify}. Let $\Omega_0$ be any of its domain of linearity that contains $2\rho$. Then
$$f|_{\Omega}(y)=(f(2\rho)+\Lambda(2\rho))-\Lambda(y),$$
for some $\Lambda\in\overline{\mathfrak a_+}$. On the other hand, since $f$ is concave,
$$f(y)\leq (f(2\rho)+\Lambda(2\rho))-\Lambda(y),~\forall y\in\overline{P_+}.$$
By Theorem \ref{H-f-reduction} we get
\begin{align*}
H(\mathcal F)\geq&\ln\left(\frac1V\int_{P_+}e^{-f(y)+f(2\rho)}\pi(y)dy\right)\\
\geq&\ln\left(\frac1V\int_{P_+}e^{\Lambda(y-2\rho)}\pi(y)dy\right)=:\mathcal H({\Lambda}).
\end{align*}
It is direct to check that $\mathcal H(\cdot)$ is strictly convex and proper on $\overline{\mathfrak a_+}$. It admits a unique minimizer $\Lambda_0\in\overline{\mathfrak a_+}$. By Corollary \ref{R-TC-Lambda}, there is a $\hat G$-equivariant special $\mathbb R$-test configuration $\mathcal F_{\Lambda_0}$ with H-invariant
$$H(\mathcal F_{\Lambda_0})=\mathcal H(\Lambda_0)=\min_{\Lambda\in\overline{\mathfrak a_+}}\mathcal H(\Lambda).$$
Here the first equality follows from Theorem \ref{H-f-reduction} and fact that $\mathcal F_{\Lambda_0}$ is special. Hence we get \eqref{H-minima}. On the other hand, by Corollary \ref{equ-minimizer}, $\mathcal F_{\Lambda_0}$ is also the minimizer of $H(\cdot)$ among all filtrantions. We conclude that $\mathcal F_{\Lambda_0}$ is the semistable degeneration.


\end{proof}

We want to test the K-polystability of the central fibre. The following barycenter condition will be used.

\begin{lem}\label{property-minima}
Let $\mathcal F_{\Lambda_0}$ be the minimizer in Theorem \ref{main-thm-1}. Suppose that $(\Lambda_0,0)\in\hat{\mathfrak t}$ satisfies \eqref{wall-Lambda}-\eqref{no-wall-Lambda}. Set $\Xi_0={\rm Span}_{\mathbb R_+}\{\alpha_1,...,\alpha_{i_0}\}.$
Then
\begin{align}\label{bar-Lambda0}
\mathbf{b}(\Lambda_0):=\frac{\int_{P_+}y_ie^{\Lambda_0(y)}\pi dy}{\int_{P_+}e^{\Lambda_0(y)}\pi dy}\in2\rho+\overline{\Xi_0}.
\end{align}
\end{lem}

\begin{proof}
Let $\{\varpi_i\}_{i=1,...,r}$ be the fundamental weights with respect to $\Phi_{+,s}$. That is
$$\varpi_i(\alpha_j)=\frac12|\alpha_j|^2\delta_{ij},~1\leq i,j\leq r.$$
Hence the Weyl wall orthogonal to $\alpha_i$ is $$
W_{\alpha_i}={\rm Span}_{\mathbb R}\{\varpi_j|j=1,...,i-1,i+1,...,r\}.
$$
By \eqref{no-wall-Lambda}
we can write
$${\rm RelInt}(\cap_{i=1,...,i_0}W_i)\ni\Lambda_0=\sum_{j=i_0+1}^rc_j\varpi_j,~c_j>0.$$
Hence, $\Lambda_0$ is also an interior minima of $\mathcal H|_{\cap_{i=1,...,i_0}W_i}(\cdot)$. We have
\begin{align}\label{dH-j}
0=\frac{\partial\mathcal H}{\partial \varpi_j}(\Lambda_0)=\varpi_j(\mathbf{b}(\Lambda_0)-2\rho){\int_{P_+}e^{\Lambda_0(y)}\pi dy},~j=i_0+1,...,r.
\end{align}

On the other hand, $\Lambda_0$ is a boundary minima in the half space $\{y|\varpi_i(y)\geq0\}$ for $i=1,...,i_0$. Thus
\begin{align}\label{dH-i}
0\leq\frac{\partial\mathcal H}{\partial \varpi_i}(\Lambda_0)=\varpi_i(\mathbf{b}(\Lambda_0)-2\rho){\int_{P_+}e^{\Lambda_0(y)}\pi dy},~i=1,...,i_0.
\end{align}

Note that for any $i\in\{1,...,r\}$,
\begin{align*}
\{y|\varpi_i(y)\geq0\}={\rm Span}_{\mathbb R_{\geq0}}\{\pm\alpha_1,\pm\alpha_{i-1},\alpha_i,\pm\alpha_{i+1},...,\pm\alpha_r\}.
\end{align*}
Combining the above relation with \eqref{dH-j}-\eqref{dH-i}, we get \eqref{bar-Lambda0}.
\end{proof}

Combining with Proposition \ref{valuation-cone}, we have
\begin{prop}\label{m-semi-stab}
Suppose that $\mathcal F_{\Lambda_0}$ is the minimizer in Theorem \ref{main-thm-1} so that $\Lambda_0$ satisfies \eqref{wall-Lambda}-\eqref{no-wall-Lambda}. Then the central fibre $\mathcal X_0$ of $\mathcal F_{\Lambda_0}$ is $\hat G$-equivariantly modified K-semistable with respect to the vector field $\Lambda_0$. 
In addition, if \eqref{bar-Lambda0} is strict, i.e.
\begin{align}\label{strict-stable}
\mathbf{b}(\Lambda_0)\in2\rho+\Xi_0,
\end{align}
then $\mathcal X_{0}$ is modified K-polystable and the K\"ahler-Ricci flow \eqref{kahler-Ricci-flow} on $M$ converges to $(\mathcal X_0,\Lambda_0)$.
\end{prop}

\begin{proof}
By Theorem \ref{main-thm-1}, $\mathcal X_0$ is normal. Thus it is a $\mathbb Q$-Fano spherical variety. To compute the combinatorial data of $\mathcal X_0$ it is more convenient to realize it via a $\mathbb Z$-test configuration.  
For this purpose,we can slightly perturb $\Lambda_0$ to a rational $\Lambda_0'$ so that $\mathcal F_{\Lambda_0'}$ is a $\mathbb Z$-configuration whose central fibre is also $\mathcal X_0$. Note that by \cite[Section 2.2]{Han-Li} (see \cite[end of p.9 - beginning of p.10]{Han-Li}), $\Lambda_0'$ should be chosen in the Lie algebra of $\overline{\exp(t\Lambda_0)}$. Hence $(\Lambda_0',0)$ also satisfies \eqref{wall-Lambda}-\eqref{no-wall-Lambda}. By Proposition \ref{valuation-cone}, the valuation cone of $\mathcal X_0$ is given by \eqref{val-cone}.
On the other hand, by Proposition \ref{polytope-X0}, the moment polytope of $(\mathcal X_0,K^{-1}_{\mathcal X_0})$ is also $P_+$. Recall the group $\hat L$ given by \eqref{hat-L-levi} and denote by $Z(\hat L)$ its centre. By Lemma \ref{equi-aut} and \eqref{wall-Lambda}, $\Lambda_0\in{\rm Aut}_{\hat G}(\mathcal X_0)=Z(\hat L)$. Hence by Lemma \ref{property-minima} and \cite[Theorem 5.3]{Del3}, $(\mathcal X_0,\Lambda_0)$ is $\hat G$-equivariantly modified K-semistable and $\Lambda_0$ is the soliton vector field on $\mathcal X_0$. Note that by Corollary \ref{R-TC-Lambda}, $\mathcal F_{\Lambda_0}$ induces the $\overline{\exp(t\Lambda_0)}$-action on $\mathcal X_0$.

Now we prove the second part. If \eqref{strict-stable} holds, we can show that $(\mathcal X_0,\Lambda_0)$ is (modified) $\hat G$-uniformly Ding-polystable in the sense of \cite[Definition 5.17]{Han-Li-KRS}. Once this is proved, \cite[Theorem 6.3]{Han-Li-KRS} will imply that the modified Ding functional (with respect to $\Lambda_0$) on $\mathcal X_0$ is $\hat G$-coercive. Note that a $\mathbb Q$-Fano spherical variety always has klt singularities (cf. \cite[Section 5]{AB2}). By \cite[Theorem 3.5]{Han-Li-KRS} $\mathcal X_0$ admits a (singular) K\"ahler-Ricci soliton (with soliton vector field $\Lambda_0$). Hence $(\mathcal X_0,\Lambda_0)$ is modified K-polystable (cf. \cite[Theorem 6.3]{Han-Li-KRS}). By the uniqueness theorem \cite[Theorem 1.3]{Han-Li}, the ``polystable degeneration" is trivial. Hence $\mathcal X_{0}$ is the limiting space of \eqref{kahler-Ricci-flow}.

It remains to show the (modified) $\hat G$-uniform Ding-polystability provided \eqref{strict-stable} holds. Denote by $W_L$ the Weyl group of $\Phi_L$ (see Section 3.2.2 above). It is well-known that any $\hat G$-equivariant normal $\mathbb Z$-test configuration $\mathcal F$ corresponds to a rational, concave, $W_L$-invariant piecewise linear function $f$ on the $W_L$-invariant convex polytope $P_{\mathcal X_0}:=\cup_{w\in\Phi_{+,L}}w(P_+)$. As in the first part of the proof of Lemma \ref{L-F-lem-app}, one gets
$$L^{\rm NA}(\mathcal F)=f(2\rho).$$
On the other hand, for the $\mathbb Z$-test configuration $\mathcal F$, it holds
$$\mathcal F^{ks}R_k=\oplus_{\lambda\in\overline{kP_+}\cap\mathfrak M,[kf(\lambda/k)]\geq ks}{\rm End}(V_\lambda),~\forall k\in\mathbb N.$$
Using the Rimann-Roch formula of \cite[Section 1.4]{Pukhlikov-Khovanskii},
\begin{align*}
{\rm vol}_{e^{\Lambda_0}}(\mathcal F^{(s)})=&\lim_{k\to+\infty}\frac{n!}{k^n}\sum_{\lambda\in\overline{kP_+}\cap\mathfrak M,[kf(\lambda/k)]\geq ks}e^{\Lambda_0(\lambda/k)}\dim({\rm End}(V_\lambda))\\
=&\int_{\overline{P_+}\cap\{f\geq s\}}e^{\Lambda_0(y)}\pi dy.
\end{align*}
Taking $g=e^{\Lambda_0(y)}$ in the formula \cite[Eq. (5.46)]{Han-Li-KRS} and combining with the above relation, one concludes
\begin{align}\label{E-NA-g}
E_{\Lambda_0}^{\rm NA}(\mathcal F)=-\frac{\int_{\mathbb R}sd{\rm vol}_{e^{\Lambda_0}}(\mathcal F^{(s)})}{\int_{P_+}e^{\Lambda_0(y)}\pi dy}=\frac{\int_{P_+}fe^{\Lambda_0(y)}\pi dy}{\int_{P_+}e^{\Lambda_0(y)}\pi dy}.
\end{align}
Thus the modified non-Archimedean Ding functional
\begin{align*}
D_{\Lambda_0}^{\rm NA}(\mathcal F):=-E_{\Lambda_0}^{\rm NA}(\mathcal F)+L^{\rm NA}(\mathcal F)
=-\frac{\int_{P_+}fe^{\Lambda_0(y)}\pi dy}{\int_{P_+}e^{\Lambda_0(y)}\pi dy}+f(2\rho).
\end{align*}

Then we compute the modified non-Archimedean J-functional. Let $(\mathcal X,\mathcal L)$ be the total space of $\mathcal F$. Then it is a polarized compactification of $(\hat G\times \mathbb C^*)/(H\times\{e\})$ with moment polytope $\mathcal P_\mathcal L=\{(y,t)|0\leq t\leq f(y),~y\in P_+\}.$ The modified non-Archimedean J-functional
\begin{align*}
J_{\Lambda_0}^{\rm NA}(\mathcal F):=\frac 1{L^{\cdot n}}\mathcal L\cdot L_{\mathbb{CP}^1}^{\cdot n}-E_{\Lambda_0}^{\rm NA}(\mathcal F).
\end{align*}
By \eqref{E-NA-g} it suffices to compute the first term on the right-hand-side. To compute the intersection number, we use the method of \cite[Section 18]{Timashev-book}. Note that for any $\epsilon>0$, the Newton polytope of the ample line bundle $\mathcal L_\epsilon:=\epsilon \mathcal L+ L_{\mathbb{CP}^1}$ is $\mathcal P_\epsilon:=\epsilon \mathcal P_\mathcal L+(P_+\times\{0\})\subset\mathfrak M_\mathbb R\oplus\mathbb R$ (see Figure-1 below). By \cite[Corollary 18.28]{Timashev-book},
$$\frac1{(n+1)!}(\mathcal L_\epsilon)^{\cdot(n+1)}=\int_{\mathcal P_\epsilon}\pi dy\wedge dt=\epsilon\max f\cdot\int_{P_+}\pi dy+O(\epsilon^2),~\epsilon\to0^+.$$
Hence
$$\mathcal L\cdot L_{\mathbb{CP}^1}^{\cdot n}=n!\left.\frac{d}{d\epsilon}\right|_{\epsilon_0}(\mathcal L_\epsilon)^{\cdot(n+1)}=n!\max f\cdot\int_{P_+}\pi dy=L^{\cdot n}\cdot\max f.$$

\begin{figure}[h]
\begin{center}
\begin{tikzpicture}
\fill[color=gray!20] (0,0)--(0,2)--(1,1.8) -- (2,1.5)--(3,0.3)--(3,0)--(0,0);
\fill[color=gray!35] (0,0)--(0,0.2)--(3,0.2)--(3.1,0.18) -- (3.2,0.15)--(3.3,0.03)--(3.3,0)--(0,0);
\fill[color=gray!50] (3,0)--(3,0.2)--(3.1,0.18) -- (3.2,0.15)--(3.3,0.03)--(3.3,0)--(3,0);
\draw [very thick] (0,0) -- (3,0);
\draw [very thick] (0,-0.5) -- (3,-0.5);
\draw (-0.2,-0.5) node {\scriptsize{$P_+$}};
\draw [semithick] (0,-0.8) -- (3.3,-0.8);
\draw (-0.7,-0.8) node {\scriptsize{$(1+\epsilon)P_+$}};
\draw [semithick] (0,0) -- (3.3,0);
\draw[semithick]  (0,2)--(1,1.8) -- (2,1.5)--(3,0.3)--(3,0);
\draw[semithick]  (0,0.2)--(3,0.2)--(3.1,0.18) -- (3.2,0.15)--(3.3,0.03)--(3.3,0);
\draw[semithick]  (0,0)--(0,2);
\draw[dashed]  (3.15,0.05)--(3.15,1);
\draw[dashed]  (0,0)--(0,-0.8);
\draw[dashed]  (3,0)--(3,-0.5);
\draw[dashed]  (3.3,0)--(3.3,-0.8);
\draw (0,0) node{$\bullet$};
\draw (3.15,1.3) node{\scriptsize{\scriptsize{$O(\epsilon^2)$}}};
\draw (-0.2,0.2) node {\scriptsize{$O$}};
\draw (1,1.35) node {$\mathcal P_\mathcal L$};
\draw (1.5,2) node {\scriptsize{$f$}};
\draw (1.5,0.35) node {\scriptsize{$\epsilon\max f$}};
\draw (3.5,0.2) node {$\mathcal P_\epsilon$};
\end{tikzpicture}
\end{center}
Figure-1: The polytope $\mathcal P_\epsilon$.
\end{figure}

Thus we get
\begin{align*}
J_{\Lambda_0}^{\rm NA}(\mathcal F)=&\frac{\int_{P_+}(\max f-f)e^{\Lambda_0(y)}\pi dy}{\int_{P_+}e^{\Lambda_0(y)}\pi dy}.
\end{align*}

Recall Lemma \ref{equi-aut}. The twist $\mathcal F_{(\xi)}$ of $\mathcal F$ by an element $\xi$ in $\mathfrak{z(aut)}(\mathcal X_0)=\mathfrak{z(\hat l)}$ is associated to the function $f_\xi(y)=f(y)+\xi(y)$. Hence
\begin{align*}
J_{\Lambda_0}^{\rm NA}(\mathcal F_{(\xi)})=&\frac{\int_{P_+}(\max f_\xi-f_\xi)e^{\Lambda_0(y)}\pi dy}{\int_{P_+}e^{\Lambda_0(y)}\pi dy}.
\end{align*}
Note that $\xi\in\mathfrak{z(\hat l)}$ is $W_L$-invariant.

On the other hand, using the argument of \cite[Proposition 4.5]{LZ}, one can prove there is a constant $\epsilon_0>0$ so that
$$D_{\Lambda_0}^{\rm NA}(\mathcal F)\geq\epsilon_0J_{\Lambda_0}^{\rm NA}(\mathcal F)$$
holds for any concave $W_L$-invariant function $f$ satisfying the normalized condition\footnote{In fact, we apply the argument of \cite[Proposition 4.5]{LZ} to the convex function $u=-f$ and weight $e^{\Lambda_0(y)}\pi(y)$.}
$$\max f=f(2\rho-2\rho_L)=0.$$
Here $2\rho-2\rho_L=\sum_{\alpha\in\Phi_+\setminus\Phi_{+,L}}\alpha\in\mathfrak {z^*(l)}$. Note that $f$ can always be normalized by subtracting a $W_L$-invariant affine function. We get
$$D_{\Lambda_0}^{\rm NA}(\mathcal F)\geq\epsilon_0\inf_{\xi\in\mathfrak{z(\hat l)}}J_{\Lambda_0}^{\rm NA}(\mathcal F_{(\xi)}).$$
Hence $(\mathcal X_0,\Lambda_0)$ is modified $\hat G$-uniformly Ding-polystable.
\end{proof}

\begin{rem}
Combining \eqref{bar-Lambda0} (or \eqref{strict-stable}, respectively) with \cite[Theorem 5.3]{Del3}, one directly concludes that $(\mathcal X_0,\Lambda_0)$ is $\hat G$-equivariantly modified K-semistable (or K-polystable, respectively). Apply the arguments of \cite[Theorem 1.1]{Zhuang-2021} to the weighted $\delta$-invariant defined in \cite[Section 4.1]{Blum-Liu-Xu-Zhuang} instead of the usual one, one concludes the corresponding modified K-stability regardless the $\hat G$-action.\footnote{We thank Ziquan Zhuang for introducing us the paper \cite{Zhuang-2021}.}
\end{rem}

\begin{rem}
When $\Lambda_0\in\mathfrak{z(g)}$, we see that $\hat H_0={\rm diag}(G)\times \mathbb C^*$ and the corresponding $\mathcal F_{\Lambda_0}$ is indeed a product test configuration. In this case $\mathcal X_0=M$. 
If in addition \eqref{strict-stable} holds, then $M$ admits a K\"ahler-Ricci soliton. See also \cite[Section 5]{LZZ}.
\end{rem}

\section{Application to $SO_4(\mathbb C)$-compactifications}
In \cite[Example 5.12]{Del3}, Delcroix showed two K-unstable smooth Fano compactifications of {$\mathrm{SO}_4(\mathbb{C})$ by giving their moment polytopes\footnote{In fact, by using \cite[Theorem 9]{Timashev-Sbo}, we can check that there are three smooth Fano $\mathrm{SO}_4(\mathbb{C})$-compactifications. Furthermore, by \cite[Theorem A]{Del4} we see that one of them is K-stable and the other two are K-unstable.}.} In this section we will determine the limits of \eqref{kahler-Ricci-flow} on these $SO_4(\mathbb C)$-compactifications by using Theorem \ref{main-thm-1} and Proposition \ref{m-semi-stab}. In this way we show Theorem \ref{SO-4-exa}.

To describe the polytopes in detail, choose a coordinate on $\mathfrak a^*$ such that the basis are the generator of $\mathfrak M$. Then the positive roots are $\alpha_1=(1,-1),~\alpha_2=(1,1).$
Thus, $2\rho=(2,0)$,
\begin{eqnarray*}
\mathfrak a_+^*=\{x>y>-x\},~2\rho+\Xi=\{-2+x>y>2-x\},
\end{eqnarray*}
and $\pi(x,y)=(x-y)^2(x+y)^2.$

For both of $P_+$, the barycenter of $P_+$, the barycenter $\mathbf{b}(0) \not\in \overline {2\rho+\Xi}.$
Hence the corresponding $SO_4(\mathbb C)$-compactifications admit no K\"ahler-Einstein metrics. Moreover, The Futaki invariant vanishes since the center of
automorphisms group is finite. Hence there are also no other K\"ahler-Ricci solitons on those compactifications. It is proved in \cite{LTZ} that the K\"ahler-Ricci flow on them develops Type-II solutions.

\begin{figure}[h]
\begin{center}
\begin{tikzpicture}[scale=0.75]
\draw [dotted] (0,-2) grid[xstep=1,ystep=1] (3,3);
\draw (0,0) node{$\bullet$};
\draw (2,0) node{$\bullet$};
\draw (1.6,0.3) node{$2\rho$};
\draw [semithick] (3,3) -- (3,0) -- (3/2,-3/2) -- (0,0) -- (3,3);
\draw (2.7,2.3) node{$P_+$};
\draw (0.7,2.7) node{(1)};
\draw [very thick, -latex] (0,0) -- (1,-1);
\draw [very thick, -latex] (0,0) -- (1,1);
\end{tikzpicture}
\begin{tikzpicture}[scale=0.75]
\draw [dotted] (0,-2) grid[xstep=1,ystep=1] (3,3);
\draw (0,0) node{$\bullet$};
\draw (2,0) node{$\bullet$};
\draw (1.6,0.3) node{$2\rho$};
\draw [semithick] (3,3) -- (3,1) -- (2,-1) -- (3/2,-3/2) --(0,0) -- (3,3);
\draw (2.7,2.3) node{$P_+$};
\draw (0.7,2.7) node{(2)};
\draw [very thick, -latex] (0,0) -- (1,-1);
\draw [very thick, -latex] (0,0) -- (1,1);
\end{tikzpicture}
\end{center}
Figure-2.
\end{figure}


\emph{Case-(1).} The polytope is
\begin{eqnarray*}
P_+=\{y>-x,x> y, 2-x>0,2+y>0,3-x+y>0\}.
\end{eqnarray*}
By using a \texttt{Wolframe Mathematica 8} programma, we get the critical point of $\mathcal H(\cdot)$,
\begin{eqnarray*}
\Lambda_0=s(1,-1),~\text{where}~s\in(0.15210775,0.15210800).
\end{eqnarray*}
We see that $\Lambda_0\in\ker(\alpha_2)$. 
We can write $\mathcal X_0$ as a $\hat G/H_0$-compactification where $\hat G=SO_4(\mathbb C)\times SO_4(\mathbb C),$
and $H_0\subset \hat G$ whose Lie algebra
\begin{align*}
\mathfrak h_0=\mathbb C(\alpha_2,\alpha_2)\oplus\mathbb C(\alpha_1,0)\oplus&(\mathbb C(X_{\alpha_2},X_{\alpha_2})\oplus\mathbb C(X_{-\alpha_2},X_{-\alpha_2}))\\
\oplus&(\mathbb C(0,X_{\alpha_1})\oplus\mathbb C(X_{-\alpha_1},0)).
\end{align*}
Thus the valuation cone
$$\mathcal V(\hat G/H_0)=\{(x,y)|\alpha_2(x,y)=x+y\geq0\}.$$
The polytope of $\mathcal X_0$ remains the same as (1).

It is direct to check that
$$\alpha_2({\mathbf b}(\Lambda_0)-2\rho)>0.$$
By Proposition \ref{m-semi-stab}, we see that the limit $\mathcal X_0$ is indeed modified K-polystable with respect to $\Lambda_0$. Thus $(\mathcal X_0,\Lambda_0)$ is the desired limit.

\emph{Case-(2)}. The polytope is
\begin{eqnarray*}
P_+=\{y>-x,x> y, 2-x>0,2+y>0,3-x+y>0,5-2x+y>0\}.
\end{eqnarray*}
Again, by using a \texttt{Wolframe Mathematica 8} programma, we can check that $\mathcal H(\cdot)$
has no critical point in $2\rho+\partial\Xi$. Hence $$\Lambda_0\in2\rho+{\rm RelInt}(\Xi),$$
which lies in neither $\ker(\alpha_1)$ nor $\ker(\alpha_2)$. 
We see that the central fibre $\mathcal X_0$ is a $\hat G/H_0$-compactification with $\hat G=SO_4(\mathbb C)\times SO_4(\mathbb C),$
and $H_0\subset \hat G$ whose Lie algebra
\begin{align*}
\mathfrak h_0=\mathbb C(\Lambda_0,0)\oplus\mathbb C(\Lambda_0^\perp,\Lambda_0^\perp)\oplus(\oplus_{i=1,2}(\mathbb C(0,X_{\alpha_i})\oplus\mathbb C(X_{-\alpha_i},0))).
\end{align*}
Here $\Lambda_0^\perp$ is any vector in $\mathfrak a$ orthogonal to $\Lambda_0$. Since $\Lambda_0$ does not perpendicular to any simple root, by Proposition \ref{H0} the central fibre $\mathcal X_0$ is a horospherical variety. Hence it always admits a K\"ahler-Ricci soliton with soliton vector field $\Lambda_0$.

\section{Appendix: An algebraic proof of Corollary \ref{R-TC-Lambda}}
In this Appendix, we give an algebraic proof of Corollary \ref{R-TC-Lambda}. Recall that a $\hat G$-equivariant normal $\mathbb R$-test configuration $\mathcal F$ with reduced central fibre is special if and only if ${\rm Gr}(\mathcal F)$ is an integral ring. It suffices to show
\begin{prop}\label{prop-app}
Suppose that \eqref{s-k-G-final-reduced} holds. Then the algebra ${\rm Gr}(\mathcal F)$ defined by \eqref{GrF-def} is integral if and only if $f$ is affine on $P_+$.
\end{prop}

\begin{proof}
Assume that ${\rm Gr}(\mathcal F)$ is integral. We show that $f$ is affine. Otherwise, we can take two domains of linearity $Q_1,Q_2\subset P_+$ so that they intersect along a common facet. Take $\lambda_i\in Q_i\cap\mathfrak M_\mathbb Q$ so that the line segment $\overline{\lambda_1\lambda_2}\subset Q_1\cup Q_2$. Up to replacing $P_+$ by some $k_0P_+$, we can assume $\lambda_i\in Q_i\cap\mathfrak M$ for $i=1,2$.

Let $\sigma_i\in{\rm End}(V_{\lambda_i})$ be a highest weight vector. Then $\sigma_i$ has real weight $t^{-s_{\lambda_i}^{(1)}}$. Consequently, $\sigma_1\otimes \sigma_2\in{\rm End}(V_{\lambda_1+\lambda_2})$ has real weight $t^{-s_{\lambda_1}^{(1)}-s_{\lambda_2}^{(1)}}$. On the other hand, for $\lambda_1+\lambda_2\in\overline{2P_+}\cap\mathfrak M$, by \eqref{s-k-G-final-reduced},
$$s_{\lambda_1+\lambda_2}^{(2)}=2f(\frac12(\lambda_1+\lambda_2))>s_{\lambda_1}^{(1)}+s_{\lambda_1}^{(2)},$$
where the last inequality follows from the concavity of $f$ and the fact that $\lambda_i$'s lie in different domains of linearity. Hence $\sigma_1\cdot\sigma_2=0$ in ${\rm Gr}(\mathcal F)$. A contradiction to the assumption that ${\rm Gr}(\mathcal F)$ is integral.

Conversely, assume that $f$ is affine. We will show that ${\rm Gr}(\mathcal F)$ is integral. Otherwise, there are $(0\not=)\sigma_i\in{\rm End}(V_{\lambda_i}),i=1,2$ so that $\sigma_1\cdot\sigma_2=0$ in ${\rm Gr}(\mathcal F)$. By Lemma \ref{alg-lem}, we can assume that each $\sigma_i$ is a highest weight vector. Assume that $\lambda_i\in\overline{k_iP_+}$. Then by \eqref{s-k-G-final-reduced}, $\sigma_i$ has real weight $t^{-k_if(\lambda_i/k_i)}$. Consequently, $\sigma_1\otimes \sigma_2\in{\rm End}(V_{\lambda_1+\lambda_2})$ has real weight $t^{-k_1f(\lambda_1/k_1)-k_2f(\lambda_2/k_2)}$ with
$$k_1f(\frac{\lambda_1}{k_1})+k_2f(\frac{\lambda_2}{k_2})=(k_1+k_2)f(\frac{\lambda_1+\lambda_2}{k_2+k_2}).$$
Here we used the fact that $f$ is affine. Note that the right-hand side is just the real weight of the ${\rm End}(V_{\lambda_1+\lambda_2})$-piece in ${\rm Gr}(\mathcal F)$. Hence $\sigma_1\cdot\sigma_2\not=0$, a contradiction. The Proposition is proved.
\end{proof}


\end{document}